\documentclass[11pt,twoside]{article}

\usepackage[utf8]{inputenc}
\usepackage[backend=bibtex,doi=true,url=true, style=alphabetic]{biblatex}

\usepackage{amsmath,amsthm,amsfonts,amssymb,mathtools,mathdots,scalerel,mathrsfs,stmaryrd,cmll,bbold}
\usepackage[T1]{fontenc}
\usepackage{mlmodern, tikz-cd, quiver}
\usepackage[many]{tcolorbox}
\usepackage{xypic}
\usepackage{xcolor}

\usepackage[shortlabels]{enumitem}
\usepackage[hidelinks]{hyperref}
\usepackage{tikz, fancyhdr, xparse, xcolor, tocloft, xspace}
\usepackage[british]{babel}
\usepackage{csquotes}
\usepackage[a4paper, top=4.5cm, bottom=4.5cm, left=3.5cm, right=3.5cm, asymmetric]{geometry}

\usepackage{crossreftools}
\usepackage[textwidth=3cm, textsize=small, colorinlistoftodos]{todonotes}

\usepackage{cctitle, titlesec}
\usepackage{etoolbox}

\usetikzlibrary{nfold}

\makeatletter
\patchcmd{\ttlh@hang}{\parindent\z@}{\parindent\z@\leavevmode}{}{}
\patchcmd{\ttlh@hang}{\noindent}{}{}{}
\makeatother

\newcommand\eqdef{\coloneqq}
\newcommand\nbd{\nobreakdash-\hspace{0pt}}
\newcommand\idd[1]{\mathrm{id}_{#1}}
\newcommand\bigid[1]{\mathrm{Id}_{#1}}
\newcommand\invrs[1]{#1^{-1}}
\newcommand\slice[2]{{#1}/{\raisebox{-2pt}{$#2$}}}
\newcommand{\set}[1]{\{#1\}}
\newcommand\fun[1]{\mathsf{#1}}
\newcommand{\eqclass}[1]{[#1]}
\newcommand\quot[2]{{#1}/{\raisebox{-2pt}{$#2$}}}

\DeclareMathOperator{\supp}{\sf supp}
\newcommand{\D}{\fun{D}}
\newcommand{\Quot}{\fun{Qt}}
\newcommand{\quottf}{\kappa}
\newcommand{\Q}{\fun{Q}}

\newcommand{\tD}{\tilde{\D}} %
\newcommand{\dirac}[1]{\delta_{#1}}
\newcommand{\idmat}[1]{{\sf id}_{#1}}
\newcommand{\tp}[1]{{#1}^{\fun{T}}}

\newcommand{\kro}{\otimes}
\newcommand{\copyy}[1]{\Delta_{#1}}
\newcommand{\disc}[1]{E_{#1}}
\newcommand{\dom}[1]{\overline{#1}}

\newcommand{\cat}[1]{\mathbf{#1}}
\newcommand{\Set}{\cat{Set}}
\newcommand{\PSet}{\cat{pSet}}
\newcommand{\Cat}{\cat{Cat}}

\newcommand{\SubMat}{\mathbf{subMat}}
\newcommand{\PSubMat}{\mathbf{psubMat}}
\newcommand\restr[2]{{#1}{\raisebox{0pt}{$|_{#2}$}}}
\newcommand\incl{\hookrightarrow}
\DeclareMathOperator*{\colim}{colim}
\newcommand{\pt}{\mathbf{1}}

\newcommand\cls[1]{\mathbb{#1}}
\newcommand{\C}{\cls{C}}
\newcommand{\X}{\cls{X}}
\newcommand{\Y}{\cls{Y}}
\newcommand{\K}{\mathcal{K}}
\newcommand{\T}{\fun{T}}
\newcommand{\F}{\fun{F}}
\newcommand{\G}{\fun{G}}

\newcommand\celto{\Rightarrow}
\newcommand{\comp}[1]{*_{#1}}
\newcommand{\MonCatlax}{\cat{SymMonCat}}
\newcommand{\Mnd}{\mathrm{Mnd}}
\DeclareMathOperator{\Kl}{\mathrm{Kl}}
\newcommand{\kl}{{\scalebox{0.7}{\( \Kl \)}}}

\newcommand{\Phy}{\cat{Phy}}
\newcommand{\Comp}{\cat{Comp}}
\newcommand{\tf}{\rightsquigarrow}
\newcommand{\bbR}{\mathbb{R}}

\renewcommand{\H}{\fun{H}}
\newcommand{\Hphy}{\H_{\mathrm{phy}}}
\newcommand{\Hcomp}{\H_{\mathrm{comp}}}
\newcommand{\Hnc}{\H_{\mathrm{nc}}}

\newcommand{\scl}{.8}
\definecolor{darkMagenta}{HTML}{DE3163}
\newcommand{\tc}[1]{\textcolor{darkMagenta}{#1}}

\newtheoremstyle{ittheorem}
  {\topsep}   % ABOVESPACE
  {\topsep}   % BELOWSPACE
  {\itshape}  % BODYFONT
  {0pt}       % INDENT (empty value is the same as 0pt)
  {\bfseries} % HEADFONT
  { --- }         % HEADPUNCT % Changed from --- to none
  {5pt plus 1pt minus 1pt} % HEADSPACE
  {}          % CUSTOM-HEAD-SPEC

\newtheoremstyle{itdfn}
  {\topsep}
  {\topsep}
  { }
  {0pt}
  {\bfseries}
  {}
  {5pt plus 1pt minus 1pt}
  {Definition \thmnumber{#2}{\thmnote{\normalfont\ \ %
{\sffamily(#3)}.}}}

\newtheoremstyle{itrmk}
  {0.5\topsep}
  {0.5\topsep}
  {\normalfont}
  {0pt}
  {\sffamily \itshape}
  { --- }
  {5pt plus 1pt minus 1pt}
  {}
  
\theoremstyle{ittheorem}
\newtheorem{thm}{Theorem}[section]
\newtheorem*{thm*}{Theorem}
\newtheorem{prop}[thm]{Proposition}
\newtheorem*{prop*}{Proposition}
\newtheorem{cor}[thm]{Corollary}
\newtheorem{lem}[thm]{Lemma}
\newtheorem*{cor*}{Corollary}
\theoremstyle{itdfn}
\newtheorem{dfn}[thm]{}
\theoremstyle{itrmk}
\newtheorem{rmk}[thm]{Remark}
\newtheorem{comm}[thm]{Comment}

\setlist{leftmargin=20pt,itemsep=0pt,parsep=0pt,topsep=1ex}

\titleformat{\section}
 {\Large\bfseries}{\thesection}{1em}{}

\titleformat{\subsection}
 {\normalsize\bfseries}{\thesubsection.}{1em}{}

\tikzstyle{strings}=[baseline={([yshift=-.5ex]current bounding box.center)}]

\tikzset{every picture/.append style={scale=1}, transform shape, strings}

\tikzset{%
symbol/.style={%
draw=none,
every to/.append style={%
edge node={node [sloped, allow upside down, auto=false]{$#1$}}}
}
}

\usetikzlibrary{shapes.geometric}
\usetikzlibrary{patterns}
\usetikzlibrary{fit}
\usetikzlibrary{positioning}
\usetikzlibrary{calc}
\usetikzlibrary{arrows}
\usetikzlibrary{decorations.markings}
\usetikzlibrary{decorations.pathreplacing}
\usetikzlibrary{shapes}

\pgfdeclarelayer{nodelayer}
\pgfdeclarelayer{edgelayer}
\pgfsetlayers{edgelayer,nodelayer,main}

\tikzset{simple/.style={}}
\tikzset{nothing/.style={outer sep=-3.4pt}}
\tikzset{map/.style={draw,fill=white, rectangle}}
\tikzstyle{filled}=[-, fill=black]

\tikzset{dot/.style={thick, fill=black, circle, scale=1, inner sep = .05cm}}

\tikzset{circ/.style={
shape=circle, inner sep=2pt, thick, draw}}

\tikzset{ox/.style={draw, scale=0.9,minimum height=.1cm,circle,append after command={
[shorten >=\pgflinewidth, shorten <=\pgflinewidth,]
(\tikzlastnode.north west) edge (\tikzlastnode.south east)
(\tikzlastnode.north east) edge (\tikzlastnode.south west) } } }

\tikzstyle{none}=[inner sep=-1pt]
\tikzstyle{circle}=[shape=circle,draw]

\tikzstyle{onehalfcircle}=[shape=circle, scale=1.5, draw]
\tikzstyle{twocircle}=[shape=circle, scale=2, draw]
\tikzstyle{black}=[shape=circle, fill=black, draw]

\tikzset{wires/.style={}}

\newcommand\runtitle{compositional account of generalized reversible computing}
\newcommand\runauthor{chanavat and srinivasan}

\title{A Compositional Account of \\ Generalized Reversible Computing}
\author{Cl\'emence Chanavat and Priyaa Varshinee Srinivasan}

\institution{Tallinn University of Technology, Estonia}

\begin{document}

\maketitle
\begin{center}
	\vspace{-12pt}
	\begin{minipage}[t]{.95\textwidth}
		\small\textsc{Abstract.} We develop a compositional framework for generalized reversible computing using copy-discard categories and resource theories. 
		We introduce partitioned matrices between partitioned sets as subdistribution matrices which preserve the equivalence relation of its domain.
		We model computational and physical transformations as subdistribution matrices over the category of sets and partitioned matrices on partitioned sets, respectively. 
		We show that the interactions between the physical, and computational transformations are governed by an aggregation functor whose functoriality and monoidality we deduce from general principles of the formal theory of monads.
		We study the associated copy-discard structures, in particular, general conditions for determinism and partial invertibility.
		We then define several notions of entropies that we use to state and prove the fundamental theorem of generalized reversible computing. 
	\end{minipage}
	
	\vspace{1pt}

	\begin{minipage}[t]{0.95\textwidth}
		\setcounter{tocdepth}{2}
		\tableofcontents
	\end{minipage}
\end{center}

\makeaftertitle

\section{Introduction}

Landauer's principle \cite{landauer1961principle} bridges information theory and thermodynamics by providing a theoretical lower limit for energy dissipation during a computation.
It states:
{\em ``any logically irreversible manipulation of information, such as the erasure of a bit or the merging of two computation paths, must be accompanied by a corresponding increase in entropy (heat dissipation) in the environment.''}

Current computers dissipate almost all of their consumed electrical energy as heat 
leading to dramatic effects at the scale of supercomputers and data centers. For example, 
around 40\% of total energy consumption of a data center is related to its cooling 
systems \cite{winick2025reducing}. Reversible computing founded upon Landauer's principle 
aims to minimize the energy cost of computation by making every computational step 
reversible. In spite of the debates regarding the connection between logical and physical 
reversibility \cite{bennett2003notes}, Landauer's principle and reversible computing 
remain an active area of research and development for their potential energy advantage.

In a recent article, Frank observed the gap between the engineering principles and the theory of reversible computing \cite{Fra18}. Reversible computing as described by the current theories is overly restrictive excluding operations utilized by engineers. 
In order to address this disconnect, Frank proposed a theory of \emph{generalized reversible computing} (GRC), which is concerned with identifying the minimal requirements at the logical level for computational operations to be reversible at the physical level.

A computational system \cite{Fra18} is described using a partitioned set, or equivalently, a set equipped with an equivalence relation. 
The individual elements of the set represent the \emph{physical states} of the system and all the computationally indistinguishable states are considered as a single partition. 
Thus, each partition with one or more physical states is considered a \emph{computational state}. 
The {\it context of a computational system} is described by means of a probability distribution over the set of physical states, which in turn induces a probability distribution over the set of computational states. 

A computational process transforms a computational system from one state to another probabilistically. 
At the physical level, a computational process is non-entropy ejecting if the {\em non-computational entropy} (a term which represents uncertainty purely at the physical level, see \cite[Definition 7]{Fra18}) stays the same or decreases throughout the process. 
Then, the {\em fundamental theorem of generalized reversible computing} \cite[Theorem 6]{Fra18} formally states the equivalence between non-entropy-ejecting processes and a weakened notion of logical reversibility called \emph{conditional reversibility}, a reversibility which is defined over a subset of computational states rather than on the entire state space. 
This formal equivalence has been set up using set-theoretic foundations. 

A fundamental aspect that is not evident in the set-theoretical framework of GRC is when two conditionally reversible processes can be composed to obtain new conditionally reversible process. 
Compositionality is a crucial feature of computation, as circuits are constructed by composing gates both sequentially and in parallel. 
Hence, the main goal of the article is to use category theory to formulate a compositional model of generalized reversible computing described in \cite{Fra18} and to present the formal equivalence using this framework. 
As such, this article lies at the crossroads of physics, computer science, and category theory.

\subsection*{Overview: compositional theory of generalized reversible computing}

In this work, we construct three main mathematical interfaces, namely: 

\begin{enumerate}[1.]
\item the category of partitioned sets and partitioned matrices called $\PSubMat$ representing computation at the physical level (Definition \ref{Defn:partitioned-matrix}),  
\item the category of sets and subdistribution matrices called $\SubMat$ representing computation at the logical level (Definition \ref{dfn:subdistribution_matrices}), and 
\item a monoidal functor $\Q \colon (\PSubMat, \kro, \pt) \to (\SubMat, \kro, \pt)$ relating physical and logical levels of computation, (Definition \ref{dfn:aggregation}),
\end{enumerate}

We construct the categories of physical and computational contexts called \( \Phy \) (Definition \ref{defn:comp})  and \( \Comp \) (Definition \ref{defn:phy}) as subcategories of the coslice categories \( \slice{\pt}{\PSubMat} \) and \( \slice{\pt}{\SubMat} \) respectively.  
The objects of the coslice categories are \emph{contexts}, that is, probability distributions on (partitioned) sets, and morphisms are (partitioned) subdistribution matrices. 
%We define {\bf Phy} to be subcategory of \( \slice{\pt}{\PSubMat} \) containing only those partitioned matrices which preserve (Shannon) entropy -- we call such partitioned matrices as {\bf closed physical transformations} \cite[Theorem 1]{Fra18}. 

The functor $\Q \colon \PSubMat \to \SubMat$ restricts to a functor $\Q \colon \Phy \to \Comp$ by mapping a physical context to its corresponding computational context, and a physical transformation to its corresponding computational transformation. 
This functor is essentially surjective, ensuring that every computational context has a corresponding physical context. 
It is also full (for locally finite partitioned sets), so that every computational transformation can be modelled by a physical transformation (Lemma \ref{lem: Q_is_full}).

Following \cite{Fra18}, we capture physical reversibility in $\Phy$ using non-entropy ejecting processes. 
For this, one requires notions of physical, computational, and non-computational entropies (Definition \ref{defn:entropies}). 
Given a physical context, its physical entropy is simply the entropy of its underlying distribution. 
Its computational entropy is the entropy of the aggregated distribution under $\Q$. 
The difference between these two quantities is the \emph{non-computational entropy}, which could also been defined as the physical entropy conditioned on the computational information \cite[Theorem 2]{Fra18}. 
A physical transformation (in $\Phy$) is said to be non-entropy ejecting when it does not increase non-computational entropy, in other words, not lose computational entropy (Definition \ref{def: non-entropy-ejecting}).

On the computational side, we use the notion of \emph{conditional reversibility}, \cite[Definition 16]{Fra18}: a deterministic computational transformation in $\Comp$ is conditionally reversible if it admits a partial inverse (details in Subsection \ref{Subsec:computational-rev}).

Finally, Proposition \ref{prop:nee_iff_condrev} states the equivalence of physical and computational (logical) reversibility -- a physical transformation whose corresponding computational transformation is deterministic, is non-entropy ejecting if and if only if the computational transformation conditionally reversible.

Let us delve in more detail the categorical instruments that will support the above-sketched compositional model.

\subsection*{Categorical wiring behind partitioned matrices}

We rely first on classical tools of categorical probability theory and most notably \emph{copy-discard categories} \cite{CoG99, ChJ19} as our basis to develop the compositional framework of generalized reversible computing. 
Copy-discard categories are based on monoidal categories, also known as process theories, which already provide a unified framework for modeling a wide range of systems—automata, cryptographic, concurrent, quantum, and probabilistic, to name a few.

A copy-discard category is a monoidal category \( (\cls{C}, \otimes, I) \) with extra structural morphisms allowing to produce copies of an object \( \Delta \colon X \to X \otimes X \), and to delete (or discard) an object \( \varepsilon \colon X \to I \).
Those morphisms are required to satisfy a certain number of axioms: for instance, copying an object followed by deleting one of the copies must amount to doing nothing.   
Already at this level of generality, there are natural definitions for when a morphism is, for instance, \emph{deterministic}, \emph{total}, or a \emph{partial isomorphism}, see Definition \ref{defn:Deterministic, total, and quasi-total}.

The prototypical example of a copy-discard category is \( \Set \) with its monoidal structure given by the cartesian product.
The copy function \( \Delta \colon X \to X \times X \) sends \( x \in X \) to the pair \( (x, x) \), and the discard function \( \epsilon \colon X \to \pt \) is the unique function to the terminal set. 
In the category \( \Set \), every function is deterministic and total.
Partial isomorphisms are precisely isomorphisms.
More generally, a copy-discard category in which every morphism is deterministic is a cartesian restriction category \cite{CoS02b,Nester2024} which is used as a categorical axiomatization of deterministic partial processes. 

A more probabilistic example arises when one considers the category in which objects are again sets, but this time, a morphism from a set \( X \) to a set \( Y \) is a \emph{subdistribution matrix} of shape \( (X, Y) \), that is, a matrix \( M \) 
indexed by \( X \times Y \) and whose \( x \)\nbd th row is a (finitely supported) subdistribution on \( Y \).
Such matrices can be composed using matrix multiplication.
This forms the category \( \SubMat \) of subdistribution matrices\footnote{Also called substochastic matrices}, which is also a copy-discard category, see Subsection \ref{subsec: subdistribution matrices} for more detail. 

There is a canonical functor from the category of sets to the category of subdistribution matrices, sending a set to itself, and a function \( f \colon X \to Y \) to the matrix of shape \( (X, Y) \) whose entry at \( (x, y) \) is \( 1 \) if \( y = f(x) \) and \( 0 \) otherwise.
It turns out that any deterministic morphism of the category \( \SubMat \) is of this form (see Lemma \ref{lem:deterministic_iff_quasidirac}).
This example is the canonical one of the more general trend that uses Kleisli categories of certain monads on a category \( \C \) to add ``non-deterministic'' morphisms to \( \C \) \cite{fritz2020markov,DiM23}.

Instead of using sets, in this article we are concerned with modelling computational systems using \emph{partitioned sets}. 
We think of a partitioned set \( (X, \sim) \) as some large physical system, and each equivalence class \( \eqclass{x} \) of an element \( x \in X \) as a computational unit of this physical system. As such, the category \( \PSet \) of partitioned sets and functions that respect the equivalence relations plays a central role in this work, see Definition \ref{defn:partitioned-set}.
Since the category \( \PSet \) admits products, it is a copy-discard category, where, again, all functions are deterministic and total --- but we are interested in modelling non-deterministic physical transformations.
Thus, the real main character of the article is \( \PSubMat \), the category of {\em partitioned matrices} on partitioned sets (Definition \ref{defn:category-partitioned-matrices}). 

To define partitioned matrices, we note that given a partitioned set $(X,\sim)$, we get a canonical equivalence relation on $\D X$ where $\D X$ is the set of probability distributions over $X$. 
Given a partitioned set $(X,\sim)$, we say two subdistributions $p$ and $q$ on $X$ are equivalent if the total probability assigned to each partition is the same under $p$ and $q$, see Definition \ref{defn:equivalent-subdistribution}. 
Then, a subdistribution matrix \( M \) on the sets \( (X, \sim) \) and \( (Y, \sim) \) is {\it partitioned} if whenever \( p \) and \( q \) are equivalent subdistribution on \( (X, \sim) \), $pM$ and $qM$ are equivalent subdistributions on \( (Y, \sim) \), see Definition \ref{defn:partitioned-matrices}.
This means that a partitioned matrix - representing a physical transformation - will send computationally equivalent states in $\quot X \sim$ to computationally equivalent states in $\quot Y \sim$. 
Thus, for a transformation of shape \( (X, Y) \) to be partitioned, it needs to respect both the physical system, and the underlying computational system.

Now there is a quotient functor from \( \PSet \) to \( \Set \). This functor maps a partitioned set \( (X, \sim) \) to its quotient \( \quot{X}{\sim} \).
The condition for a matrix to be partitioned turns out to be precisely what is needed to extend the quotient functor to a functor \( \Q \colon \PSubMat \to \SubMat \), by summing the entries of a partitioned matrix of shape \( (X, Y) \) along each partition of \( X \). We call \( \Q \) the \emph{aggregation functor}, see Definition \ref{dfn:aggregation}.
The main compositional result of the article is as follows (Corollary \ref{cor:aggregation_functor}).
\begin{cor*}
    The construction sending a partitioned matrix \( M \) to its aggregated matrix \( \Q M \) defines a strong monoidal functor
    \begin{equation*}
        \Q \colon (\PSubMat, \kro, \pt) \to (\SubMat, \kro, \pt).
    \end{equation*}
\end{cor*}
As stated earlier, the category \( \SubMat \) is in fact the Kleisli category of the subdistribution monad on \( \Set \).
Similarly, we prove (Proposition \ref{prop:Kleisli_is_matrix}) that the category \( \PSubMat \) is the Kleisli category of the subdistribution monad, this time on \( \PSet \).
Thus, our Corollary states two things: 
\begin{enumerate}
	\item the Kleisli categories of the subdistributions monads, on sets and partitioned sets, are monoidal, and
	\item sending a partitioned matrix to its aggregated matrix is a functorial operation that respects the monoidal structure.
\end{enumerate}
The above two facts are obtained at once by a direct application of the formal theory of monads \cite{street1972formal}.
A more detailed overview is given in Section \ref{sec: mathematical_backend}.

\subsection*{Resource theories and functors}

The final piece of our compositional framework is the resource-theoretic description of the equivalence stated in Proposition \ref{prop:nee_iff_condrev}.

Resource theories \cite{Quantum-resource-1,Quantum-resource-2,Gour-review} are used in physics to model systems in which certain operations are considered to be {\it free of cost} among the set of all operations. 
For example, placing a glass of chilled water at room temperature warms up the water to the ambient temperature at no extra cost. 
In this context, operations that bring the water temperature into equilibrium with the ambient temperature are considered free.  
In order to produce a {\it resourceful state}  --- for example, a glass of chilled water ---  one requires non-free transformations, such as a fridge, which consumes electricity. 
Resource theories have been successfully deployed to study, among other examples, thermodynamical \cite{delRio,Lostaglio-thermo}, and quantum systems \cite{Review-entanglement,Dynamical-entanglement,Review-coherence}. 

A resource theory is formally described as a category with a chosen subcategory of {\it free transformations}, see \cite[Definition 3.3]{coecke2016resource} and \cite[Definition 3.1]{Resource-Monotone}. 
In the context of the article, we consider two resource theories for a computational system:
\begin{enumerate}
    \item the physical level \( \Phy \), whose free transformations \( \Phy_{\sf f} \) are the non-entropy-ejecting transformation whose underlying computations (under the image of $\Q$) are deterministic (Definition \ref{dfn:resource-theory-phy});
    \item the computational level \( \Comp \), whose free transformations \( \Comp_{\sf f} \) are the conditionally reversible computations (Definition \ref{dfn:resource-theory-comp}).
\end{enumerate} 
A \emph{transformation of resource theories} is a functor which preserves the free transformations \cite[Definition 3.7]{Resource-Monotone}.
If a transformation maps only free transformations to free transformations, then we call it \emph{resource-reflecting}.

Then, the fundamental theorem of GRC can be described as a resource-reflecting transformations between the resource theories of physical and computational transformations as in Theorem \ref{thm: fundamental_theorem}:
\begin{thm*}
    The aggregation functor 
    \begin{equation*}
        \Q \colon (\Phy,\Phy_{\sf f}) \to (\Comp, \Comp_{\sf f})	
    \end{equation*}
    is a resource reflecting transformation of resource theories. 
\end{thm*} 

\subsection*{Related work} % Work in progress

{\bf Category Theory}: The notion of partial reversibility for `deterministic' processes have been already studied in category theory with known applications to reversible computing. Cockett and Lack introduced restriction categories \cite{CoS02} as categorical axiomatization of partial processes --- processes which are defined on a subset of their domain. 

In his PhD thesis, Brett Giles  explores a flavor of restriction categories, called the (discrete) inverse categories, as a compositional setting for studying reversible computing and its relationship with quantum computing \cite{BrettThesis}. Multiple models of reversible programming languages based on inverse categories have been studied in \cite{ambramsky2005structural, kaarsgaard2017inverse, heunen2018reversible}, and based on groupoids in \cite{CaA16, CJS22, CKS22, carette2024compositional}. With applications to optimization problems, inverse categories have been used to model classes of reversible circuits, see \cite{CCS18, CoC19}.

While these approaches apply to deterministic reversible processes, physical thermodynamics is inherently non-deterministic since it operates on collections of states rather than isolated deterministic points. In this work, we model computations as non-deterministic in order to establish the equivalence of thermodynamically reversible processes and logically reversible computations. By treating a state of a system as a probability distribution over a (partitioned) set, we formulate the necessary entropies and represent computational operations as (partitioned) subdistribution matrices acting on those distributions. Consequently, we use minimal required structure namely copy-discard categories, which are standard for modeling probabilistic processes. Rather than defining new categorical structures, we develop new models of copy-discard categories based on partitioned subdistribution matrices and appropriate functorial mappings bridging the physical and the computational worlds.

{\bf On Landauer's principle}: In physics, the single particle toy model introduced by \cite{Szi64} as a 
simplification of the Maxwell Demon's thought experiment \cite {Max71} is often used to justify the Landauer's 
principle. Hence, the philosophical debates \cite{bennett2003notes,norton2011waiting,LaR14} on the validity of 
Launder's  principle are often centered around the validity of this model. 

Abramsky and Horsman formalized the single-particle equilibrium thermodynamics as an imperative programming language termed DEMONIC \cite{AbH15}, thereby providing an operational semantics for the toy model. They showed that the second law of thermodynamics holds for every process which can be described by a program in the language, thereby validating Landauer's principle within this system. 
While Abramsky and Horsman's operational semantics uses computational invariants, our compositional framework uses (resource-reflecting) functorial mapping to demonstrate the logical consistency of the Landauer's principle.
	
\subsection*{How to read this paper?}

We have attempted to make this article accessible to a wider scientific audience by presenting the mathematical interface separate from the categorical results.
The interface of compositional generalized reversible computing is presented in Section \ref{sec: compositional_framework} and used in Section \ref{sec:fundamental-theorem}.  
Section \ref{sec: mathematical_backend} is dedicated to more abstract results supporting the interface, that the reader less familiar with the language of category theory may safely skip.  
The article starts by Section \ref{sec:preliminaries}, which reviews the basics of copy-discard categories and resource theories: none of the result presented in that Section are new, but we decided, for the sake of completeness, to include string diagrammatic proofs of certain useful results that will be used in the sequel.

\subsection*{Convention on the order of composition}

Throughout the article, we always write composition in the diagrammatic order, thus the composite of a morphism \( f \colon X \to Y \) and \( g \colon Y \to Z \) will be written \( fg \).
When working in a 2-category, we use \( - \comp{0} - \) for the horizontal composition and \( - \comp{1} - \) for the vertical composition; again, both are written diagrammatically.
Last, string diagrams are read left to right and top to bottom.

\subsection*{Acknowledgement}
The authors thank Nathanael Arkor, Robin Cockett, Elena Di Lavore, Fosco Loregian, Callum Reader, Mario Rom\'an, Pawe\l{} Soboci\'nski for their insights and feedback. We thank the reviewers for bringing to our attention the work of Abramsky and Horsman on single particle equilibrium thermodynamics and their helpful feedback.
The string diagrams in this paper have been drawn using TikZit \cite{tikzit} developed by Alex Kissinger. 
We thank Robin Kaarsgaard and Martti Karvonen for their recommendations of related works. 
This work was co-funded by the European Union and Estonian Research Council through the Mobilitas 3.0 (MOB3JD1227).

\section{Copy discard categories and resource theories}
\label{sec:preliminaries}

\subsection*{Copy discard categories}

\begin{dfn} [\tc{Cocommutative comonoid object}] \label{dfn:commcomonoid_object}
	Let \( (\C, \otimes, I, \sigma) \) be a symmetric monoidal category with braiding \( \sigma \).
    An object \( X \) of \( \C \) is a {\bf cocommutative comonoid} if there exist morphisms $\Delta \colon X \to X \otimes X$, called the {\bf copy morphism}, and $\varepsilon \colon X \to I$, called the {\bf discard morphism}, drawn respectively as
    \[ \begin{tikzpicture}[scale=\scl]
			\begin{pgfonlayer}{nodelayer}
				\node [style=circ] (0) at (-3, 2) {};
				\node [style=none] (1) at (-4, 2) {};
				\node [style=none] (2) at (-2, 2.5) {};
				\node [style=none] (3) at (-2, 1.5) {};
				\node [style=none] (4) at (-3.5, 2.25) {$X$};
				\node [style=none] (5) at (-2.25, 2.75) {$X$};
				\node [style=none] (6) at (-2.25, 1.75) {$X$};
			\end{pgfonlayer}
			\begin{pgfonlayer}{edgelayer}
				\draw [thick] (1.center) to (0);
				\draw [in=180, out=60, looseness=1.25, thick] (0) to (2.center);
				\draw [in=-180, out=-60, looseness=1.25, thick] (0) to (3.center);
			\end{pgfonlayer}
		\end{tikzpicture} \quad \text{ and } \quad 
		\begin{tikzpicture}[scale=\scl]
			\begin{pgfonlayer}{nodelayer}
				\node [style=circ] (0) at (0, 2.5) {};
				\node [style=none] (1) at (-1.75, 2.5) {};
				\node [style=none] (2) at (-1.25, 2.75) {$X$};
			\end{pgfonlayer}
			\begin{pgfonlayer}{edgelayer}
				\draw [thick] (1.center) to (0);
			\end{pgfonlayer}
		\end{tikzpicture}
		\]
		and satisfying
		\begin{itemize}
			\item the counit axiom: $\Delta (X \otimes \varepsilon) = \Delta (\varepsilon \otimes X) = \idd{X}$,
			\item the coassociativity axiom: $\Delta(X \otimes \Delta) = \Delta(\Delta \otimes X) $, and
			\item the cocommutativity axiom: $\Delta = \Delta \sigma_{X, X}$.
		\end{itemize}
		In string diagrams, we have respectively:
		\begin{align*}
		{\bf {[counit]}}& \quad\quad
		\begin{tikzpicture}[scale=\scl]
			\begin{pgfonlayer}{nodelayer}
				\node [style=circ] (2) at (-3, 1) {};
				\node [style=none] (3) at (-4, 1) {};
				\node [style=none] (4) at (-2, 1.5) {};
				\node [style=circ] (5) at (-2, 0.5) {};
				\node [style=none] (10) at (1.5, 1) {};
				\node [style=none] (11) at (3.25, 1) {};
				\node [style=none] (12) at (1, 1) {$=$};
				\node [style=none] (13) at (-1.5, 1) {$=$};
				\node [style=circ] (14) at (-0.25, 1) {};
				\node [style=none] (15) at (-1.25, 1) {};
				\node [style=none] (16) at (0.75, 0.5) {};
				\node [style=circ] (17) at (0.75, 1.5) {};
			\end{pgfonlayer}
			\begin{pgfonlayer}{edgelayer}
				\draw [thick](3.center) to (2);
				\draw [thick, in=180, out=60, looseness=1.25] (2) to (4.center);
				\draw [thick, in=-180, out=-60, looseness=1.25] (2) to (5);
				\draw [thick](10.center) to (11.center);
				\draw [thick](15.center) to (14);
				\draw [in=-180, out=-60, looseness=1.25, thick] (14) to (16.center);
				\draw [in=180, out=60, looseness=1.25, thick] (14) to (17);
			\end{pgfonlayer}
		\end{tikzpicture} \\
		{\bf {[coassoc]}}& \quad\quad
		\begin{tikzpicture}[scale=\scl]
			\begin{pgfonlayer}{nodelayer}
				\node [style=circ] (2) at (-3, 1) {};
				\node [style=none] (3) at (-4, 1) {};
				\node [style=none] (4) at (-2, 1.5) {};
				\node [style=circ] (5) at (-2, 0.5) {};
				\node [style=none] (13) at (-0.25, 1) {$=$};
				\node [style=none] (18) at (-1, 1) {};
				\node [style=none] (19) at (-1, 0) {};
				\node [style=none] (20) at (-1, 1.5) {};
				\node [style=circ] (21) at (1.25, 0.5) {};
				\node [style=none] (22) at (0.25, 0.5) {};
				\node [style=none] (23) at (2.25, 0) {};
				\node [style=circ] (24) at (2.25, 1) {};
				\node [style=none] (25) at (3.25, 0.5) {};
				\node [style=none] (26) at (3.25, 1.5) {};
				\node [style=none] (27) at (3.25, 0) {};
			\end{pgfonlayer}
			\begin{pgfonlayer}{edgelayer}
				\draw [thick] (3.center) to (2);
				\draw [thick, in=180, out=60, looseness=1.25] (2) to (4.center);
				\draw [in=-180, out=-60, looseness=1.25, thick] (2) to (5);
				\draw [thick] (4.center) to (20.center);
				\draw [in=-180, out=60, looseness=1.25, thick] (5) to (18.center);
				\draw [in=-180, out=-60, looseness=1.25, thick] (5) to (19.center);
				\draw [thick] (22.center) to (21);
				\draw [in=-180, out=-60, looseness=1.25, thick] (21) to (23.center);
				\draw [in=180, out=60, looseness=1.25, thick] (21) to (24);
				\draw [thick] (23.center) to (27.center);
				\draw [in=180, out=-60, looseness=1.25, thick] (24) to (25.center);
				\draw [in=180, out=60, looseness=1.25, thick] (24) to (26.center);
			\end{pgfonlayer}
		\end{tikzpicture}\\
		{\bf {[cocomm]}}& \quad\quad
		\begin{tikzpicture}[scale=\scl]
			\begin{pgfonlayer}{nodelayer}
				\node [style=circ] (2) at (-3, 1) {};
				\node [style=none] (3) at (-4, 1) {};
				\node [style=none] (4) at (-2, 1.5) {};
				\node [style=none] (5) at (-2, 0.5) {};
				\node [style=circ] (6) at (0, 1) {};
				\node [style=none] (7) at (-1, 1) {};
				\node [style=none] (8) at (0.5, 1.5) {};
				\node [style=none] (9) at (0.5, 0.5) {};
				\node [style=none] (10) at (1.5, 0.5) {};
				\node [style=none] (11) at (1.5, 1.5) {};
				\node [style=none] (12) at (-1.5, 1) {$=$};
			\end{pgfonlayer}
			\begin{pgfonlayer}{edgelayer}
				\draw [thick] (3.center) to (2);
				\draw [thick, in=180, out=60, looseness=1.25] (2) to (4.center);
				\draw [thick, in=-180, out=-60, looseness=1.25] (2) to (5.center);
				\draw [thick] (7.center) to (6);
				\draw [thick, in=180, out=60, looseness=1.25] (6) to (8.center);
				\draw [thick, in=-180, out=-60, looseness=1.25] (6) to (9.center);
				\draw [thick, in=-180, out=0] (8.center) to (10.center);
				\draw [thick, in=180, out=0, looseness=0.75] (9.center) to (11.center);
			\end{pgfonlayer}
		\end{tikzpicture}
		\end{align*}
\end{dfn}

\begin{dfn} [\tc{Copy-discard category}] 
    A {\bf copy-discard category} is a symmetric monoidal category \( (\C, \otimes, I, \sigma) \) such that every object \( X \in \C \) is equipped with a commutative comonoid structure \( (X, \Delta_X, \varepsilon_X) \) which is {\bf uniform}, that is,
    \begin{itemize}
        \item \( \Delta_I = \idd{I} \), \( \varepsilon_I = \idd{I} \), and
        \item for all pairs of objects \( X, Y \) in \( \C \), we have 
		\begin{equation*}
			\varepsilon_{X \otimes Y} = \varepsilon_X \otimes \varepsilon_Y \text{ and } \Delta_{X \otimes Y} = (\Delta_X \otimes \Delta_X) (X \otimes \sigma_{X, Y} \otimes Y) .
		\end{equation*}
        The latter equality is drawn as follows.
        \[ \begin{tikzpicture}[scale=\scl]
			\begin{pgfonlayer}{nodelayer}
				\node [style=circ] (2) at (-3, 1) {};
				\node [style=none] (3) at (-4, 1) {};
				\node [style=none] (4) at (-2, 1.5) {};
				\node [style=none] (5) at (-2, 0.5) {};
				\node [style=circ] (13) at (0.5, 0.5) {};
				\node [style=circ] (17) at (0.5, 1.5) {};
				\node [style=ox] (21) at (1.5, 1.5) {};
				\node [style=ox] (22) at (1.5, 0.5) {};
				\node [style=ox] (23) at (-0.25, 1) {};
				\node [style=none] (24) at (-1, 1) {};
				\node [style=none] (25) at (2, 0.5) {};
				\node [style=none] (26) at (2, 1.5) {};
				\node [style=none] (27) at (-4, 1.25) {$X \otimes Y$};
				\node [style=none] (28) at (-1.5, 1) {$=$};
				\node [style=none] (29) at (2.25, 1.75) {$X \otimes Y$};
				\node [style=none] (30) at (2.25, 0.75) {$X \otimes Y$.};
				\node [style=none] (31) at (0, 1.75) {$Y$};
				\node [style=none] (32) at (0, 0.25) {$X$};
			\end{pgfonlayer}
			\begin{pgfonlayer}{edgelayer}
				\draw[thick] (3.center) to (2);
				\draw [in=180, out=60, looseness=1.25, thick] (2) to (4.center);
				\draw [in=-180, out=-60, looseness=1.25, thick] (2) to (5.center);
				\draw[thick] (21) to (26.center);
				\draw[thick] (22) to (25.center);
				\draw[thick] (24.center) to (23);
				\draw [in=-180, out=60, looseness=1.25, thick] (23) to (17);
				\draw [in=-180, out=-60, looseness=1.25, thick] (23) to (13);
				\draw [bend left=45, looseness=1.25, thick] (17) to (21);
				\draw [bend right=45, thick] (13) to (22);
				\draw[thick] (17) to (22);
				\draw[thick] (13) to (21);
			\end{pgfonlayer}
		\end{tikzpicture} \]
    \end{itemize}
\end{dfn}

\noindent In the rest of the article, we write {\bf CDU} in place of uniform copy-discard.
For the rest of this section, we fix a CDU category \( (\C, \otimes, I, \sigma) \).

\begin{dfn}[\tc{Domain of definition}] \label{dfn:domain_def}
	Let \( f \colon X \to Y \) be a morphism in \( \C \).
	The {\bf domain of definition} of \( f \) is the morphism \( \dom{f} \colon X \to X \) defined by \( \dom{f} \eqdef \Delta_X (X \otimes (f \varepsilon_Y))  \).
		\[ \begin{tikzpicture}[scale=\scl]
			\begin{pgfonlayer}{nodelayer}
				\node [style=circ] (2) at (-3, 1) {};
				\node [style=none] (3) at (-4, 1) {};
				\node [style=circ, scale=3] (4) at (-2, 0.5) {};
				\node [style=none] (5) at (-2, 1.5) {};
				\node [style=none] (6) at (-2, 0.5) {$f$};
				\node [style=circ] (7) at (-1, 0.5) {};
				\node [style=none] (8) at (-1, 1.5) {};
				\node [style=none] (9) at (-3.5, 1.25) {$X$};
				\node [style=none] (10) at (-1.25, 0.75) {$Y$};
			\end{pgfonlayer}
			\begin{pgfonlayer}{edgelayer}
				\draw[thick] (3.center) to (2);
				\draw[thick] [in=-180, out=-60, looseness=1.25] (2) to (4);
				\draw [in=180, out=60, looseness=1.25, thick] (2) to (5.center);
				\draw[thick] (4) to (7);
				\draw[thick] (5.center) to (8);
			\end{pgfonlayer}
		\end{tikzpicture}
		\]
\end{dfn}

\begin{lem} \label{Lem: DoD tensor}
	Let $f\colon A \to B$ and $g\colon C \to D$ be two morphisms in \( \C \).
	Then $\dom{f \otimes g} = \dom{f} \otimes \dom{g}$ \end{lem}
\begin{proof}
\[ 
	\begin{tikzpicture}[scale=\scl]
		\begin{pgfonlayer}{nodelayer}
			\node [style=none] (1) at (-2.75, 1.25) {};
			\node [style=circ] (4) at (-2, 1.25) {};
			\node [style=none] (16) at (-0.5, 2) {};
			\node [style=circ] (19) at (0.25, 0.75) {};
			\node [style=circ, scale=2.5] (26) at (-0.5, 0.75) {};
			\node [style=none] (27) at (-0.5, 0.75) {$g$};
			\node [style=none] (28) at (0.25, 2) {};
			\node [style=none] (29) at (-2.75, 2) {};
			\node [style=circ] (30) at (-2, 2) {};
			\node [style=none] (31) at (-0.5, 2.75) {};
			\node [style=circ] (32) at (0.25, 1.5) {};
			\node [style=circ, scale=2.5] (33) at (-0.5, 1.5) {};
			\node [style=none] (34) at (-0.5, 1.5) {$f$};
			\node [style=none] (35) at (0.25, 2.75) {};
		\end{pgfonlayer}
		\begin{pgfonlayer}{edgelayer}
			\draw [style=thick, in=-180, out=60, looseness=1.25] (4) to (16.center);
			\draw [style=thick, in=-180, out=-45, looseness=1.25] (4) to (26);
			\draw [style=thick] (26) to (19);
			\draw [style=thick] (16.center) to (28.center);
			\draw [style=thick] (1.center) to (4);
			\draw [style=thick, in=-180, out=75, looseness=1.25] (30) to (31.center);
			\draw [style=thick, in=-180, out=-60] (30) to (33);
			\draw [style=thick] (33) to (32);
			\draw [style=thick] (31.center) to (35.center);
			\draw [style=thick] (29.center) to (30);
		\end{pgfonlayer}
	\end{tikzpicture} 
	= 
	\begin{tikzpicture}[scale=\scl]
		\begin{pgfonlayer}{nodelayer}
			\node [style=none] (1) at (-2.75, 1.25) {};
			\node [style=circ] (4) at (-2, 1.25) {};
			\node [style=none] (16) at (-0.5, 2) {};
			\node [style=circ] (19) at (0.25, 0.75) {};
			\node [style=circ, scale=2.5] (26) at (-0.5, 0.75) {};
			\node [style=none] (27) at (-0.5, 0.75) {$g$};
			\node [style=none] (28) at (0.25, 2) {};
			\node [style=none] (29) at (-2.75, 2) {};
			\node [style=circ] (30) at (-2, 2) {};
			\node [style=none] (31) at (-0.5, 3) {};
			\node [style=circ] (32) at (0.25, 2.5) {};
			\node [style=circ, scale=2.5] (33) at (-0.5, 2.5) {};
			\node [style=none] (34) at (-0.5, 2.5) {$f$};
			\node [style=none] (35) at (0.25, 3) {};
			\node [style=none] (36) at (-1.5, 1.5) {};
			\node [style=none] (37) at (-1, 2.5) {};
		\end{pgfonlayer}
		\begin{pgfonlayer}{edgelayer}
			\draw [style=thick, in=-180, out=60, looseness=1.25] (4) to (16.center);
			\draw [style=thick, in=-180, out=-45, looseness=1.25] (4) to (26);
			\draw [style=thick] (26) to (19);
			\draw [style=thick] (16.center) to (28.center);
			\draw [style=thick] (1.center) to (4);
			\draw [style=thick, in=-180, out=75, looseness=1.25] (30) to (31.center);
			\draw [style=thick] (33) to (32);
			\draw [style=thick] (31.center) to (35.center);
			\draw [style=thick] (29.center) to (30);
			\draw [style=thick, in=-180, out=-75] (30) to (36.center);
			\draw [style=thick, in=-180, out=0] (36.center) to (37.center);
			\draw [style=thick] (37.center) to (33);
		\end{pgfonlayer}
	\end{tikzpicture} 
	= 
	\begin{tikzpicture}[scale=\scl]
		\begin{pgfonlayer}{nodelayer}
			\node [style=none] (29) at (-2.75, 3) {};
			\node [style=circ] (30) at (-2, 3) {};
			\node [style=none] (31) at (-0.5, 3.5) {};
			\node [style=circ] (32) at (0.25, 2.5) {};
			\node [style=circ, scale=2.5] (33) at (-0.5, 2.5) {};
			\node [style=none] (34) at (-0.5, 2.5) {$f$};
			\node [style=none] (35) at (0.25, 3.5) {};
			\node [style=none] (36) at (-1, 2.5) {};
			\node [style=none] (37) at (-1, 2.5) {};
			\node [style=none] (38) at (-2.75, 1.5) {};
			\node [style=circ] (39) at (-2, 1.5) {};
			\node [style=none] (40) at (-0.5, 2) {};
			\node [style=circ] (41) at (0.25, 1) {};
			\node [style=circ, scale=2.5] (42) at (-0.5, 1) {};
			\node [style=none] (43) at (-0.5, 1) {$g$};
			\node [style=none] (44) at (0.25, 2) {};
			\node [style=none] (45) at (-1, 1) {};
			\node [style=none] (46) at (-1, 1) {};
		\end{pgfonlayer}
		\begin{pgfonlayer}{edgelayer}
			\draw [style=thick, in=-180, out=60] (30) to (31.center);
			\draw [style=thick] (33) to (32);
			\draw [style=thick] (31.center) to (35.center);
			\draw [style=thick] (29.center) to (30);
			\draw [style=thick, in=-180, out=-60, looseness=1.25] (30) to (36.center);
			\draw [style=thick] (37.center) to (33);
			\draw [style=thick, in=-180, out=60] (39) to (40.center);
			\draw [style=thick] (42) to (41);
			\draw [style=thick] (40.center) to (44.center);
			\draw [style=thick] (38.center) to (39);
			\draw [style=thick, in=-180, out=-60, looseness=1.25] (39) to (45.center);
			\draw [style=thick] (46.center) to (42);
		\end{pgfonlayer}
	\end{tikzpicture} \]
\end{proof}

\begin{dfn}[\tc{Deterministic, total, and quasi-total}]\label{defn:Deterministic, total, and quasi-total}
    Let \( f \colon X \to Y \) be a morphism in \( \C \).
	We say that \( f \) is
    \begin{itemize}
        \item {\bf deterministic} if the copy morphism is natural, that is, $\Delta_X (f \otimes f)  = f \Delta_Y$.
        \item {\bf total} if the discard morphism is natural, that is, \( f \varepsilon_Y = \varepsilon_X \).
        \item {\bf quasi-total} if $f = \dom{f}f$.
   \end{itemize}
   In string diagrams, we have respectively:
   \begin{align*}
	{\bf {[det]}}& \quad \quad 
		 \begin{tikzpicture}[scale=\scl]
			\begin{pgfonlayer}{nodelayer}
				\node [style=circ] (2) at (-3, 1) {};
				\node [style=none] (3) at (-4, 1) {};
				\node [style=circ, scale=3] (4) at (-2, 0.5) {};
				\node [style=circ, scale=3] (5) at (-2, 1.5) {};
				\node [style=none] (6) at (-2, 0.5) {$f$};
				\node [style=none] (7) at (-1, 0.5) {};
				\node [style=none] (9) at (-3.5, 1.25) {$X$};
				\node [style=none] (10) at (-1.5, 0.75) {$Y$};
				\node [style=none] (12) at (-2, 1.5) {$f$};
				\node [style=none] (13) at (-1, 1.5) {};
				\node [style=none] (14) at (-1.5, 1.75) {$Y$};
			\end{pgfonlayer}
			\begin{pgfonlayer}{edgelayer}
				\draw[thick] (3.center) to (2);
				\draw[thick] [in=-180, out=-60, looseness=1.25] (2) to (4);
				\draw[thick] [in=180, out=60, looseness=1.25] (2) to (5);
				\draw[thick] (4) to (7.center);
				\draw[thick] (5) to (13.center);
			\end{pgfonlayer}
		\end{tikzpicture} 
		= 
		\begin{tikzpicture}[scale=\scl]
			\begin{pgfonlayer}{nodelayer}
				\node [style=circ] (2) at (-3, 1) {};
				\node [style=circ,scale=3] (3) at (-3.75, 1) {};
				\node [style=none] (4) at (-2, 0.5) {};
				\node [style=none] (5) at (-2, 1.5) {};
				\node [style=none] (10) at (-2.25, 0.75) {$Y$};
				\node [style=none] (12) at (-3.75, 1) {$f$};
				\node [style=none] (14) at (-2.25, 1.75) {$Y$};
				\node [style=none] (15) at (-4.5, 1) {};
			\end{pgfonlayer}
			\begin{pgfonlayer}{edgelayer}
				\draw[thick] (3) to (2);
				\draw[thick] [in=-180, out=-60, looseness=1.25] (2) to (4.center);
				\draw [in=180, out=60, looseness=1.25, thick] (2) to (5.center);
				\draw[thick] (15.center) to (3);
			\end{pgfonlayer}
		\end{tikzpicture} \\
	 {\bf {[tot]}} & \quad \quad 
			\begin{tikzpicture}[scale=\scl]
				\begin{pgfonlayer}{nodelayer}
					\node [style=circ] (2) at (-3, 1) {};
					\node [style=circ, scale=3] (3) at (-3.75, 1) {};
					\node [style=none] (12) at (-3.75, 1) {$f$};
					\node [style=none] (15) at (-4.5, 1) {};
				\end{pgfonlayer}
				\begin{pgfonlayer}{edgelayer}
					\draw[thick] (3) to (2);
					\draw[thick] (15.center) to (3);
				\end{pgfonlayer}
			\end{tikzpicture} = 
			\begin{tikzpicture}[scale=\scl]
				\begin{pgfonlayer}{nodelayer}
					\node [style=circ] (0) at (2, 1) {};
					\node [style=none] (1) at (0.5, 1) {};
				\end{pgfonlayer}
				\begin{pgfonlayer}{edgelayer}
					\draw[thick] (1.center) to (0);
				\end{pgfonlayer}
			\end{tikzpicture} \\
	{\bf {[q-tot]}} & \quad \quad f = 
			\begin{tikzpicture}[scale=\scl]
				\begin{pgfonlayer}{nodelayer}
					\node [style=circ] (2) at (-3, 1) {};
					\node [style=none] (3) at (-4, 1) {};
					\node [style=circ, scale=3] (4) at (-2, 0.5) {};
					\node [style=none] (6) at (-2, 0.5) {$f$};
					\node [style=circ] (7) at (-1, 0.5) {};
					\node [style=none] (9) at (-3.5, 1.25) {$X$};
					\node [style=none] (10) at (-1.25, 0.75) {$Y$};
					\node [style=circ, scale=3] (11) at (-2, 1.5) {};
					\node [style=none] (12) at (-2, 1.5) {$f$};
					\node [style=none] (13) at (-1, 1.5) {};
				\end{pgfonlayer}
				\begin{pgfonlayer}{edgelayer}
					\draw[thick] (3.center) to (2);
					\draw[thick] [in=-180, out=-60, looseness=1.25] (2) to (4);
					\draw[thick] (4) to (7);
					\draw[thick] (11) to (13);
					\draw[thick] [in=-180, out=60, looseness=1.25] (2) to (11);
				\end{pgfonlayer}
			\end{tikzpicture}
   \end{align*}
\end{dfn} 

\begin{lem} \label{lem:det_are_q_tot}
	Let \( f \colon X \to Y \) be a deterministic morphism in \(\C\). 
	Then \( f \) is quasi-total.
\end{lem}
\begin{proof}
Using string diagrams, we have
	\[ 
	\begin{tikzpicture}[scale=\scl]
		\begin{pgfonlayer}{nodelayer}
			\node [style=circ] (2) at (-3, 1) {};
			\node [style=none] (3) at (-4, 1) {};
			\node [style=circ, scale=3] (4) at (-2, 0.5) {};
			\node [style=none] (6) at (-2, 0.5) {$f$};
			\node [style=circ] (7) at (-1, 0.5) {};
			\node [style=none] (9) at (-3.5, 1.25) {$X$};
			\node [style=none] (10) at (-1.25, 0.75) {$Y$};
			\node [style=circ, scale=3] (11) at (-2, 1.5) {};
			\node [style=none] (12) at (-2, 1.5) {$f$};
			\node [style=none] (13) at (-1, 1.5) {};
		\end{pgfonlayer}
		\begin{pgfonlayer}{edgelayer}
			\draw[thick] (3) to (2);
			\draw[thick] [in=-180, out=-60, looseness=1.25] (2) to (4);
			\draw[thick] (4) to (7);
			\draw[thick] (11) to (13);
			\draw[thick] [in=-180, out=60, looseness=1.25] (2) to (11);
		\end{pgfonlayer}
	\end{tikzpicture}
	\stackrel{\textbf{[det]}}{=} 
	\begin{tikzpicture}[scale=\scl]
		\begin{pgfonlayer}{nodelayer}
			\node [style=circ] (2) at (-3, 1) {};
			\node [style=circ,scale=3] (3) at (-3.75, 1) {};
			\node [style=circ] (4) at (-2, 0.5) {};
			\node [style=none] (5) at (-2, 1.5) {};
			\node [style=none] (10) at (-2.25, 0.75) {$Y$};
			\node [style=none] (12) at (-3.75, 1) {$f$};
			\node [style=none] (14) at (-2.25, 1.75) {$Y$};
			\node [style=none] (15) at (-4.5, 1) {};
		\end{pgfonlayer}
		\begin{pgfonlayer}{edgelayer}
			\draw[thick] (3) to (2);
			\draw[thick] [in=-180, out=-60, looseness=1.25] (2) to (4);
			\draw [in=180, out=60, looseness=1.25, thick] (2) to (5);
			\draw[thick] (15.center) to (3);
		\end{pgfonlayer}
	\end{tikzpicture} 
	\stackrel{\textbf{[counit]}}{=} 
	\begin{tikzpicture}[scale=\scl]
		\begin{pgfonlayer}{nodelayer}
			\node [style=none] (2) at (-3, 1) {};
			\node [style=circ, scale=3] (3) at (-3.75, 1) {};
			\node [style=none] (12) at (-3.75, 1) {$f$};
			\node [style=none] (15) at (-4.5, 1) {};
		\end{pgfonlayer}
		\begin{pgfonlayer}{edgelayer}
			\draw [style=thick] (15) to (3);
			\draw [style=thick] (3) to (2);
		\end{pgfonlayer}
	\end{tikzpicture} \]
\end{proof}

\begin{lem} \label{Lem: total restriction is identity}
	Let $f \colon X \to Y$ be a total morphism in \( \C \).
	Then $\dom{f} = \idd{X}$.
\end{lem}
\begin{proof}
	\[ \begin{tikzpicture}[scale=\scl]
		\begin{pgfonlayer}{nodelayer}
			\node [style=circ] (2) at (-3, 1) {};
			\node [style=none] (3) at (-4, 1) {};
			\node [style=circ, scale=3] (4) at (-2, 0.5) {};
			\node [style=none] (5) at (-2, 1.5) {};
			\node [style=none] (6) at (-2, 0.5) {$f$};
			\node [style=circ] (7) at (-1, 0.5) {};
			\node [style=none] (8) at (-1, 1.5) {};
			\node [style=none] (9) at (-3.5, 1.25) {$X$};
			\node [style=none] (10) at (-1.25, 0.75) {$Y$};
		\end{pgfonlayer}
		\begin{pgfonlayer}{edgelayer}
			\draw[thick] (3.center) to (2);
			\draw[thick] [in=-180, out=-60, looseness=1.25] (2) to (4);
			\draw [in=180, out=60, looseness=1.25, thick] (2) to (5.center);
			\draw[thick] (4) to (7);
			\draw[thick] (5.center) to (8);
		\end{pgfonlayer}
	\end{tikzpicture} 
	\stackrel{\textbf{[tot]}}{=} 
	\begin{tikzpicture}[scale=\scl]
		\begin{pgfonlayer}{nodelayer}
			\node [style=circ] (2) at (-3, 1) {};
			\node [style=none] (3) at (-4, 1) {};
			\node [style=none] (4) at (-2, 1.5) {};
			\node [style=circ] (5) at (-2, 0.5) {};
		\end{pgfonlayer}
		\begin{pgfonlayer}{edgelayer}
			\draw [thick] (3.center) to (2);
			\draw [thick, in=180, out=60, looseness=1.25] (2) to (4.center);
			\draw [thick, in=-180, out=-60, looseness=1.25] (2) to (5);
		\end{pgfonlayer}
	\end{tikzpicture} 
	= 
	\begin{tikzpicture}[scale=\scl]
		\begin{pgfonlayer}{nodelayer}
			\node [style=none] (10) at (1.5, 1) {};
			\node [style=none] (11) at (3.25, 1) {};
		\end{pgfonlayer}
		\begin{pgfonlayer}{edgelayer}
			\draw [thick] (10.center) to (11.center);
		\end{pgfonlayer}
	\end{tikzpicture}
	\]
\end{proof} 

\begin{lem}\label{lem:det_tot_composition}
	Let $f, g$ be two morphisms in \( \C \).
	Then, 
	\begin{enumerate}
		\item if $f$ and $g$ are deterministic, then $fg$ is deterministic;
		\item if $f$ and $g$ are total, then $fg$ is total. 
		\item if $g$ is total and $fg$ is total, then $f$ is total.
	\end{enumerate}
\end{lem}
\begin{proof}
Suppose first that $f$ and $g$ are deterministic. Then,
\[ 
	\begin{tikzpicture}[scale=\scl]
		\begin{pgfonlayer}{nodelayer}
			\node [style=none] (1) at (-3.5, 2) {};
			\node [style=circ] (2) at (-1.25, 2) {};
			\node [style=none] (3) at (-0.25, 2.5) {};
			\node [style=none] (4) at (-0.25, 1.5) {};
			\node [style=circ, scale=2.5] (5) at (-2.75, 2) {};
			\node [style=circ, scale=2.5] (6) at (-2, 2) {};
			\node [style=none] (7) at (-2.75, 2) {$f$};
			\node [style=none] (8) at (-2, 2) {$g$};
		\end{pgfonlayer}
		\begin{pgfonlayer}{edgelayer}
			\draw [thick, in=-180, out=60, looseness=1.25] (2) to (3.center);
			\draw [thick, in=180, out=-60, looseness=1.25] (2) to (4.center);
			\draw [thick] (1.center) to (5);
			\draw [thick] (5) to (6);
			\draw [thick] (6) to (2);
		\end{pgfonlayer}
	\end{tikzpicture}
	= 
	\begin{tikzpicture}[scale=\scl]
		\begin{pgfonlayer}{nodelayer}
			\node [style=none] (1) at (-3.5, 2) {};
			\node [style=circ] (2) at (-2, 2) {};
			\node [style=circ, scale=2.5] (3) at (-1.25, 2.5) {};
			\node [style=circ, scale=2.5] (4) at (-1.25, 1.5) {};
			\node [style=circ, scale=2.5] (5) at (-2.75, 2) {};
			\node [style=none] (8) at (-0.75, 2.5) {};
			\node [style=none] (9) at (-0.75, 1.5) {};
			\node [style=none] (11) at (-2.75, 2) {$f$};
			\node [style=none] (13) at (-1.25, 1.5) {$g$};
			\node [style=none] (14) at (-1.25, 2.5) {$g$};
		\end{pgfonlayer}
		\begin{pgfonlayer}{edgelayer}
			\draw [thick] [in=-180, out=60, looseness=1.25] (2) to (3);
			\draw [thick][in=180, out=-60, looseness=1.25] (2) to (4);
			\draw [thick](1.center) to (5);
			\draw [thick](3) to (8.center);
			\draw [thick](4) to (9.center);
			\draw [thick](5) to (2);
		\end{pgfonlayer}
	\end{tikzpicture}
	=
	\begin{tikzpicture}[scale=\scl]
		\begin{pgfonlayer}{nodelayer}
			\node [style=none] (1) at (-3, 2) {};
			\node [style=circ] (2) at (-2, 2) {};
			\node [style=circ, scale=2.5] (3) at (-1.25, 2.5) {};
			\node [style=circ, scale=2.5] (4) at (-1.25, 1.5) {};
			\node [style=circ, scale=2.5] (8) at (-0.5, 2.5) {};
			\node [style=circ, scale=2.5] (9) at (-0.5, 1.5) {};
			\node [style=none] (13) at (-1.25, 1.5) {$f$};
			\node [style=none] (14) at (-1.25, 2.5) {$f$};
			\node [style=none] (15) at (0.25, 2.5) {};
			\node [style=none] (16) at (-0.5, 1.5) {$g$};
			\node [style=none] (17) at (-0.5, 2.5) {$g$};
			\node [style=none] (18) at (0.25, 1.5) {};
		\end{pgfonlayer}
		\begin{pgfonlayer}{edgelayer}
			\draw [thick, in=-180, out=60, looseness=1.25] (2) to (3);
			\draw [thick, in=180, out=-60, looseness=1.25] (2) to (4);
			\draw [style=thick](3) to (8);
			\draw [style=thick](4) to (9);
			\draw [style=thick] (1.center) to (2);
			\draw [style=thick] (8) to (15.center);
			\draw [style=thick] (9) to (18.center);
		\end{pgfonlayer}
	\end{tikzpicture}
\]
The second statement is proven similarly. 
Finally, assume that $fg$ and $g$ are total. 
Then, 
\(
	\begin{tikzpicture}[scale=\scl]
		\begin{pgfonlayer}{nodelayer}
			\node [style=circ] (0) at (0, 3) {};
			\node [style=none] (1) at (-2, 3) {};
			\node [style=circ, scale=2.5] (2) at (-1, 3) {};
			\node [style=none] (3) at (-1, 3) {$f$};
		\end{pgfonlayer}
		\begin{pgfonlayer}{edgelayer}
			\draw [style=thick] (1.center) to (2);
			\draw [style=thick] (2) to (0);
		\end{pgfonlayer}
	\end{tikzpicture} 
	= 
	\begin{tikzpicture}[scale=\scl]
		\begin{pgfonlayer}{nodelayer}
			\node [style=circ] (0) at (0, 3) {};
			\node [style=none] (1) at (-2, 3) {};
			\node [style=circ, scale=2.5] (2) at (-1.5, 3) {};
			\node [style=none] (3) at (-1.5, 3) {$f$};
			\node [style=circ, scale=2.5] (4) at (-0.75, 3) {};
			\node [style=none] (5) at (-0.75, 3) {$g$};
		\end{pgfonlayer}
		\begin{pgfonlayer}{edgelayer}
			\draw [style=thick] (1.center) to (2);
			\draw [style=thick] (2) to (4);
			\draw [style=thick] (4) to (0);
		\end{pgfonlayer}
	\end{tikzpicture} 
	= 
	\begin{tikzpicture}[scale=\scl]
		\begin{pgfonlayer}{nodelayer}
			\node [style=circ] (0) at (0, 3) {};
			\node [style=none] (1) at (-2, 3) {};
		\end{pgfonlayer}
		\begin{pgfonlayer}{edgelayer}
			\draw [style=thick] (1.center) to (0);
		\end{pgfonlayer}
	\end{tikzpicture}
\).
This concludes the proof.
\end{proof}

\noindent The proof of the following result is left as an exercise to the reader.
\begin{lem} \label{lem:deterministic_tensor}
	Let $f, g$ be two morphisms in \( \C \).
	Then 
	\begin{enumerate}
		\item if $f$ and $g$ are deterministic, then $f \otimes g$ is deterministic;
		\item if $f$ and $g$ are total, then $f \otimes g$ is total;
		\item if $f$ and $g$ are quasi-total, then $f \otimes g$ is quasi-total.
	\end{enumerate}
\end{lem}

\begin{dfn} [\tc{Partial isomorphism}] \label{defn:partial-inverse}
    Let \( f \colon X \to Y \) be a morphism in \( \C \).
    We say that \( f \) is a {\bf partial isomorphism} if there exists a morphism \( f^r \colon Y \to X \) such that 
    \[ f f^r  = \dom{f} \quad \text{ and } \quad  f^r f = \dom{f^r}. \]
    In that case, \( f^r \) is called a {\bf partial inverse} of \( f \).
\end{dfn}

Note that, if \( f^r \) is a partial inverse of \( f \), then by symmetry of the definition, \( f \) is a partial inverse of \( f^r \).

\begin{rmk}\label{Rmk:total-partial-isos}
	 If $f$ is total, then by Lemma \ref{Lem: total restriction is identity} we have that $\dom{f}=\idd{}$. This makes $f$ an isometry, that is, $ff^r = \dom{f} = \idd{}$. If $f^r$ is also a total map, then $f^r f = \dom{f^r}  = \idd{} $, thus $f$ is an isomorphism. 
\end{rmk}

\begin{rmk} A CDU in which every morphism is deterministic is a cartesian restriction category \cite{CoS02b,Nester2024}. A CDU in which every morphism is deterministic and has a partial inverse is also a discrete cartesian restriction category, hence also a (discrete) inverse category \cite[Theorem 5.2.6.]{BrettThesis}.
\end{rmk}

\noindent In a CDU category, partial isomorphisms generally do not compose in general. 
However, they do compose when they are deterministic, see \cite[Lemma 3.5.2 - (iii)]{BrettThesis}, a result which we reproduce next.

\begin{lem}\label{lem:partial_iso_compose}  
	Let $f \colon X \to Y$ and $g \colon X \to Y$ be deterministic partial isomorphisms in \( \C \) with deterministic partial inverse.
	Then $fg \colon X \to Z$ is a deterministic partial isomorphism. 
\end{lem}
\begin{proof}
	That \( fg \) is deterministic follows from Lemma \ref{lem:det_tot_composition}.
	Let $f^r$ and $g^r$ be partial inverses of $f$ and $g$ respectively. 
	We show that $fg$ is a partial isomorphism with partial inverse $g^r f^r$. 
	That is, we need to show that 
	\begin{equation*}
		fgg^rf^r = \dom{fg} \quad\text{ and } \quad g^rf^rfg = \dom{g^rf^r}.
	\end{equation*}
	Now \( fg g^r f^r = f \dom{g} f^r \) and by Lemma \ref{lem:det_are_q_tot}, \( f \) is quasi-total.
	Thus,
	\begin{align*}
		&f \dom{g} f^r =
		\begin{tikzpicture}[scale=\scl]
			\begin{pgfonlayer}{nodelayer}
				\node [style=none] (0) at (-3.5, 2) {};
				\node [style=circ, scale=2.5] (1) at (-3, 2) {};
				\node [style=circ] (2) at (-2.25, 2) {};
				\node [style=circ, scale=2.5] (3) at (-1.5, 2.5) {};
				\node [style=circ, scale=2.5] (4) at (-1.5, 1.5) {};
				\node [style=circ] (5) at (-0.75, 1.5) {};
				\node [style=none] (6) at (-0.75, 2.5) {};
				\node [style=none] (8) at (-3, 2) {$f$};
				\node [style=none] (9) at (-1.5, 2.5) {$f^r$};
				\node [style=none] (10) at (-1.5, 1.5) {$g$};
			\end{pgfonlayer}
			\begin{pgfonlayer}{edgelayer}
				\draw [style=thick] (0.center) to (1);
				\draw [style=thick] (1) to (2);
				\draw [style=thick, in=-180, out=75, looseness=1.25] (2) to (3);
				\draw [style=thick, in=-180, out=-75, looseness=1.25] (2) to (4);
				\draw [style=thick] (3) to (6.center);
				\draw [style=thick] (4) to (5);
			\end{pgfonlayer}
		\end{tikzpicture}
		\stackrel{\textbf{[det]}}= 
		\begin{tikzpicture}[scale=\scl]
			\begin{pgfonlayer}{nodelayer}
				\node [style=none] (0) at (-3, 2) {};
				\node [style=circ] (2) at (-2.25, 2) {};
				\node [style=circ, scale=2.5] (3) at (-1.5, 2.5) {};
				\node [style=circ, scale=2.5] (4) at (-1.5, 1.5) {};
				\node [style=none] (8) at (-1.5, 2.5) {$f$};
				\node [style=circ, scale=2.5] (11) at (-0.75, 2.5) {};
				\node [style=circ, scale=2.5] (12) at (-0.75, 1.5) {};
				\node [style=circ] (13) at (0.25, 1.5) {};
				\node [style=none] (14) at (0.25, 2.5) {};
				\node [style=none] (9) at (-0.75, 2.5) {$f^r$};
				\node [style=none] (10) at (-0.75, 1.5) {$g$};
				\node [style=none] (15) at (-1.5, 1.5) {$f$};
			\end{pgfonlayer}
			\begin{pgfonlayer}{edgelayer}
				\draw [style=thick, in=-180, out=75, looseness=1.25] (2) to (3);
				\draw [style=thick, in=-180, out=-75, looseness=1.25] (2) to (4);
				\draw [style=thick] (3) to (11);
				\draw [style=thick] (4) to (12);
				\draw [style=thick] (12) to (13);
				\draw [style=thick] (11) to (14.center);
				\draw [style=thick] (0.center) to (2);
			\end{pgfonlayer}
		\end{tikzpicture} 
		\stackrel{\textbf{[q-tot]}}= 
		\begin{tikzpicture}[scale=\scl]
			\begin{pgfonlayer}{nodelayer}
				\node [style=none] (0) at (-3, 2) {};
				\node [style=circ] (2) at (-2.25, 2) {};
				\node [style=circ, scale=2.5] (4) at (-1.5, 1.5) {};
				\node [style=circ, scale=2.5] (12) at (-0.75, 1.5) {};
				\node [style=circ] (13) at (0.25, 1.5) {};
				\node [style=none] (10) at (-0.75, 1.5) {$g$};
				\node [style=none] (15) at (-1.5, 1.5) {$f$};
				\node [style=circ] (17) at (-1.25, 2.5) {};
				\node [style=none] (18) at (0.25, 3) {};
				\node [style=circ, scale=2.5] (19) at (-0.5, 2) {};
				\node [style=none] (21) at (-0.5, 2) {$f$};
				\node [style=circ] (22) at (0.25, 2) {};
			\end{pgfonlayer}
			\begin{pgfonlayer}{edgelayer}
				\draw [style=thick, in=-180, out=-75, looseness=1.25] (2) to (4);
				\draw [style=thick] (4) to (12);
				\draw [style=thick] (12) to (13);
				\draw [style=thick] (0.center) to (2);
				\draw [style=thick, in=-180, out=75] (17) to (18.center);
				\draw [style=thick, in=-180, out=-75, looseness=1.25] (17) to (19);
				\draw [style=thick] (19) to (22);
				\draw [style=thick, in=180, out=75, looseness=1.25] (2) to (17);
			\end{pgfonlayer}
		\end{tikzpicture} \\
		&\stackrel{\textbf{[assoc]}}= 
		\begin{tikzpicture}[scale=\scl]
			\begin{pgfonlayer}{nodelayer}
				\node [style=circ] (2) at (-2.25, 2) {};
				\node [style=circ, scale=2.5] (4) at (-1.5, 1.5) {};
				\node [style=circ, scale=2.5] (12) at (-0.75, 1.5) {};
				\node [style=circ] (13) at (0, 1.5) {};
				\node [style=none] (10) at (-0.75, 1.5) {$g$};
				\node [style=none] (15) at (-1.5, 1.5) {$f$};
				\node [style=circ, scale=2.5] (17) at (-1.25, 2.5) {};
				\node [style=circ] (18) at (0, 2.5) {};
				\node [style=none] (19) at (-1.25, 2.5) {$f$};
				\node [style=circ] (20) at (-3, 2.75) {};
				\node [style=none] (21) at (-2, 3.5) {};
				\node [style=none] (22) at (0, 3.5) {};
				\node [style=none] (23) at (-3.75, 2.75) {};
			\end{pgfonlayer}
			\begin{pgfonlayer}{edgelayer}
				\draw [style=thick, in=-180, out=-75, looseness=1.25] (2) to (4);
				\draw [style=thick] (4) to (12);
				\draw [style=thick] (12) to (13);
				\draw [style=thick, in=180, out=75, looseness=1.25] (2) to (17);
				\draw [style=thick] (17) to (18);
				\draw [style=thick] (23.center) to (20);
				\draw [style=thick, in=180, out=-60] (20) to (2);
				\draw [style=thick, in=-180, out=60, looseness=1.25] (20) to (21.center);
				\draw [style=thick] (21.center) to (22.center);
			\end{pgfonlayer}
		\end{tikzpicture} 
		\stackrel{\textbf{[comm]}} = 
		\begin{tikzpicture}[scale=\scl]
			\begin{pgfonlayer}{nodelayer}
				\node [style=circ] (2) at (-2.25, 2) {};
				\node [style=circ, scale=2.5] (4) at (-1.25, 2.5) {};
				\node [style=circ, scale=2.5] (12) at (-0.5, 2.5) {};
				\node [style=circ] (13) at (0.25, 2.5) {};
				\node [style=none] (10) at (-0.5, 2.5) {$g$};
				\node [style=none] (15) at (-1.25, 2.5) {$f$};
				\node [style=circ, scale=2.5] (17) at (-1.25, 1.5) {};
				\node [style=circ] (18) at (0.25, 1.5) {};
				\node [style=none] (19) at (-1.25, 1.5) {$f$};
				\node [style=circ] (20) at (-3, 2.75) {};
				\node [style=none] (21) at (-2, 3.5) {};
				\node [style=none] (22) at (0.25, 3.5) {};
				\node [style=none] (23) at (-3.75, 2.75) {};
			\end{pgfonlayer}
			\begin{pgfonlayer}{edgelayer}
				\draw [style=thick, in=-180, out=-60, looseness=1.25] (2) to (4);
				\draw [style=thick] (4) to (12);
				\draw [style=thick] (12) to (13);
				\draw [style=thick, in=180, out=60, looseness=1.25] (2) to (17);
				\draw [style=thick] (17) to (18);
				\draw [style=thick] (23.center) to (20);
				\draw [style=thick, in=180, out=-60] (20) to (2);
				\draw [style=thick, in=-180, out=60, looseness=1.25] (20) to (21.center);
				\draw [style=thick] (21.center) to (22.center);
			\end{pgfonlayer}
		\end{tikzpicture} 
		= 
		\begin{tikzpicture}[scale=\scl]
			\begin{pgfonlayer}{nodelayer}
				\node [style=circ] (2) at (-2.25, 2) {};
				\node [style=circ, scale=2.5] (4) at (-1.25, 2.5) {};
				\node [style=circ, scale=2.5] (12) at (-0.5, 2.5) {};
				\node [style=circ] (13) at (0.25, 2.5) {};
				\node [style=none] (10) at (-0.5, 2.5) {$g$};
				\node [style=none] (15) at (-1.25, 2.5) {$f$};
				\node [style=circ, scale=2.5] (17) at (-1.25, 1.5) {};
				\node [style=circ] (18) at (0.25, 1.5) {};
				\node [style=none] (19) at (-1.25, 1.5) {$f$};
				\node [style=circ] (20) at (-3, 2.75) {};
				\node [style=none] (21) at (-2, 3.5) {};
				\node [style=none] (22) at (0.25, 3.5) {};
				\node [style=none] (23) at (-3.75, 2.75) {};
			\end{pgfonlayer}
			\begin{pgfonlayer}{edgelayer}
				\draw [style=thick] (4) to (12);
				\draw [style=thick] (12) to (13);
				\draw [style=thick] (17) to (18);
				\draw [style=thick] (23.center) to (20);
				\draw [style=thick, in=180, out=-60] (20) to (2);
				\draw [style=thick, in=-180, out=60, looseness=1.25] (20) to (21.center);
				\draw [style=thick] (21.center) to (22.center);
				\draw [style=thick, in=-180, out=60, looseness=1.50] (2) to (4);
				\draw [style=thick, in=180, out=-60, looseness=1.25] (2) to (17);
			\end{pgfonlayer}
		\end{tikzpicture} \\
		&\stackrel{\textbf{[q-tot]}}= 
		\begin{tikzpicture}[scale=\scl]
			\begin{pgfonlayer}{nodelayer}
				\node [style=circ, scale=2.5] (4) at (-2, 2.25) {};
				\node [style=circ, scale=2.5] (12) at (-1.25, 2.25) {};
				\node [style=circ] (13) at (-0.5, 2.25) {};
				\node [style=none] (10) at (-1.25, 2.25) {$g$};
				\node [style=none] (15) at (-2, 2.25) {$f$};
				\node [style=circ] (20) at (-3, 2.75) {};
				\node [style=none] (21) at (-2, 3.25) {};
				\node [style=none] (22) at (-0.5, 3.25) {};
				\node [style=none] (23) at (-3.75, 2.75) {};
			\end{pgfonlayer}
			\begin{pgfonlayer}{edgelayer}
				\draw [style=thick] (4) to (12);
				\draw [style=thick] (12) to (13);
				\draw [style=thick] (23.center) to (20);
				\draw [style=thick, in=-180, out=60, looseness=1.25] (20) to (21.center);
				\draw [style=thick] (21.center) to (22.center);
				\draw [style=thick, in=-165, out=-60, looseness=1.25] (20) to (4);
			\end{pgfonlayer}
		\end{tikzpicture}
	\end{align*}
	This proves the first equation for all pair of composable deterministic partial isomorphisms \( (f, g) \).
	Instantiating this equation with the pair of deterministic partial isomorphisms \( (g^r, f^r) \), we get the second one.
	This concludes the proof. 
\end{proof}
                        
\begin{lem} \label{lem:partial_iso_tensor}
	Let \( f, g \) be two partial isomorphisms in \( \C \).
	Then \( f \otimes g \) is a partial isomorphism.
\end{lem}
\begin{proof}
	Denote by $f^{r}$ and $g^{r}$ partial inverses of \( f \) and \( g \) respectively. 
	Then,
	\[ (f \otimes g)(f^{r} \otimes g^{r}) = ff^{r} \otimes gg^{r} = \dom{f} \otimes \dom{g} = \dom{f \otimes g}, \]
	where the last step uses Lemma \ref{Lem: DoD tensor}. 
	Similarly, $(f^{r} \otimes g^{r})(f \otimes g) = \dom{f^{r} \otimes g^{r}}$.
	Thus $(f^{r} \otimes g^{r})$ is a partial inverse of \( f \otimes g \).
\end{proof}

\subsection*{Resource theories}

We briefly recall the definition of resource theories. See \cite{Resource-Monotone,coecke2016resource} for a detailed account. 
In \cite{Resource-Monotone}, resource theories are called {\it partitioned categories}, a terminology we do not use in this article to avoid any ambiguity with the notion of partitioned matrices. 

\begin{dfn}[\tc{Resource theory}]\cite[Definition 3.1]{Resource-Monotone}
	A {\bf resource theory} is a pair $(\X, \X_{\sf f})$, where $\X$ is a category and $\X _{\sf f}$ is a chosen wide\footnote{This means that $\X _{\sf f}$ contains all of the objects of $\X$ but not necessarily all of the morphisms.} subcategory, called the subcategory of {\bf free transformations}. 
\end{dfn}

\noindent The objects of $\X$ are interpreted as {\bf resources} and the morphisms as be {\bf resource transformations}. 
The subcategory includes all objects and those transformations which are designated to be {\bf free}. 

\begin{dfn}[\tc{Functor of resource theories}] \cite[Definition 3.7]{Resource-Monotone} 
	A {\bf functor of resource theories} 
	\begin{equation*}
		\fun{F} \colon (\X, \X_{\sf f}) \to (\Y, \Y_{\sf f})
	\end{equation*}
	is a functor $\fun{F}\colon \X \to \Y$ that preserves free transformations, that is, \( \fun{F}(\X_{\sf f}) \subseteq \Y_{\sf f} \).
	We say furthermore that \( \fun{F} \) is {\bf resource reflecting} if for all morphisms \( f \) of \( \X \), \( \fun{F}(f) \in \Y_{\sf f} \) implies \( f \in \X_{\sf f} \). 
\end{dfn}

\begin{figure}[ht]
	\centering
	\includegraphics[width=0.4\textwidth]{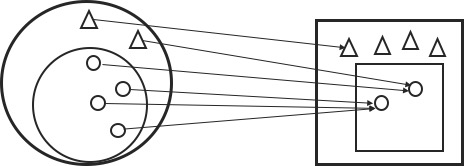} 
	\caption{Schematic for functor of resource theories}
	\label{Fig: pSMC functor}
\end{figure}

\noindent The drawing in Figure \ref{Fig: pSMC functor} is a schematic of a functor of resource theories. 
The triangles represent non-free transformations, and the hollow circles represent free transformations. 
As one can see, such a functor does not necessarily preserve the non-free transformations, unless it is resource reflecting.

\section{Compositional framework for generalized reversible computing}
\label{sec: compositional_framework}

In this section, we develop the compositional framework for generalized reversible computing. 
Our basic building bloc is that of partitioned sets, which leads to the notion of partitioned matrices.

\subsection{Partitioned sets}

\begin{dfn}[\tc{Partitioned set}] \label{defn:partitioned-set}
	A {\bf partitioned set} \( (X, \sim) \) is a set \( X \) together with an equivalence relation \( \sim \) on \( X \).
	A morphism \( f \colon (X, \sim) \to (Y, \sim) \) is a function \( f \colon X \to Y \) that respects the equivalence relation.
	We call \( \PSet \) the category of partitioned sets and their morphisms.
\end{dfn}

\noindent  The category \( \PSet \) is symmetric monoidal with respect to the cartesian product, see Corollary \ref{cor:pset_monoidal},.
If \( (X, \sim) \) and \( (Y, \sim) \) are partitioned sets, then \( (X, \sim) \times (Y, \sim) \) is the partitioned set \( (X \times Y, \sim) \) where 
\begin{equation*}
	(x, y) \sim (x', y') \quad\text{ if and only if }\quad x \sim x' \text{ and } y \sim y'.
\end{equation*}
We denote by \( \pt \eqdef \set{*} \) the monoidal unit, that is, the terminal object of \( \PSet \), which is a singleton set with its only possible equivalence relation.

\noindent As it is the case of any category with finite products, \( \PSet \) is a CDU category where, if \( (X, \sim) \) is a partitioned set, the copy-discard structure is given by
\begin{align*}
	\Delta_{(X, \sim)} &\colon x \mapsto (x, x) \\
	\varepsilon_{(X, \sim)}& \colon x \mapsto *. 
\end{align*}

In fact, by universal property of the cartesian product, the copy-discard structure defines natural transformations
\begin{align*}
	\Delta &\colon \PSet \to \PSet \times \PSet,\\
	\varepsilon &\colon \PSet \to \PSet. 
\end{align*}
As a consequence, we get the following results.
\begin{lem} \label{lem:copy_discard_pset}
	In the category of partitioned sets, the following holds:
	\begin{enumerate}
		\item All morphisms are deterministic.
		\item All morphisms are total.
		\item Partial isomorphisms are precisely isomorphisms. 
	\end{enumerate}
\end{lem}

\subsection{Subdistribution matrices}
\label{subsec: subdistribution matrices}

\noindent The statistical context of a computational system lies at the core of generalized reversible computing, wherein reversibility is defined over the subset of computational states with non-zero probability. However, the category $\PSet$ does not model such a notion of {\em partial reversibility} since all morphisms are total, see Remark \ref{Rmk:total-partial-isos}. To describe non-determinism, that is, the statistical context of a computational system, and partial reversibility, we generalize the category of \( \PSet \) further by considering {\bf partitioned matrices} as its morphisms. 
We begin by defining subdistribution matrices.

\begin{dfn} [\tc{Matrix}] \label{dfn:matrix}
    Let \( X \) and \( Y \) be sets.
    A {\bf matrix of shape \( (X, Y) \)} is a function of type \( M \colon X \times Y \to \mathbb{R} \). 
    For any \( x \in X \), the {\bf \( x \)\nbd th row} of \( M \) is the function 
	\[ M_{x, -} \colon Y \to \mathbb{R}, \] 
	and for any \( y \in Y \), the {\bf \( y \)\nbd th column} of \( M \) is the function 
	\[ M_{-, y} \colon X \to \mathbb{R}. \]
\end{dfn}

\begin{dfn} [\tc{Support of a matrix}]
    Let \( M \) be a matrix of shape \( (X, Y) \).
    The {\bf support of \( M \)} is the subset of \( X \) defined by
    \begin{equation*}
        \supp M \eqdef \set{x \in X \mid \exists y \in Y, M_{x, y} \neq 0 }.
    \end{equation*}
\end{dfn}

\begin{rmk}
    For any $x \in X$, \( x \in \supp M \) if and only if \( M_{x, -} \) is not the constant function, \( (y \mapsto 0) \).
\end{rmk}

\begin{dfn} [\tc{Finitely supported function}] 
    Let \( X \) be a set and \( f \colon X \to \mathbb{R} \) be a function.
    The {\bf support of \( f \)} is the set \[ \supp f \eqdef \set{ x \in X \mid f(x) \neq 0}. \]    We say that \( f \) is {\bf finitely supported} if it has a finite support.
\end{dfn}

\begin{rmk}
    The support of a function \( f \colon X \to \mathbb{R} \) is the support of its associated matrix of shape \( (X, \pt) \). 
\end{rmk}

\begin{dfn} [\tc{Identity matrix}]
    Let \( X \) be a set. 
    The {\bf identity matrix on \( X \)} is the matrix \( \idmat{X} \) of shape \( (X,X) \) defined by
	\begin{equation*}
		(\idmat{X})_{x, x'} \eqdef
		\begin{cases}
			1 &\text{if } x = x', \\
			0 &\text{if } x \neq x'.
		\end{cases}
	\end{equation*}
\end{dfn}

\begin{dfn} [\tc{Product of matrices}]
    Let \( X, Y, Z \) be sets, \( M \) and \( N \) be matrices of shape \( (X, Y) \) and \( (Y, Z) \) respectively such that for each \( x \in X \), the row of \( M_{x,-} \) is finitely supported.
    The {\bf product of \( M \) and \( N \)} is the matrix \( M N \) of shape \( (X, Z) \) defined by
    \begin{equation*}
        (x, z) \mapsto \sum_{y \in Y} M_{x, y} N_{y, z}.
    \end{equation*}
\end{dfn}

\begin{rmk} \label{rmk:product_matrix_unital_assoc}
    Since for any $x \in X$, the row \( M_{x,-} \) is finitely supported, there are finitely many values for which \( M_{x, y} \) is non-zero, thus \( (MN)_{x, z} \) is well-defined.
    Furthermore, it is clear that the product of matrices is associative and unital with the identity matrix. 
\end{rmk}

\begin{dfn} [\tc{Transpose of a matrix}]
    Let \( M \) be a matrix of shape \( (X, Y) \).
    The {\bf transpose of \( M \)}, written \( \tp M \), is the matrix of shape \( (Y, X) \) whose \( x \)\nbd column is the \( x \)\nbd row of \( M \), that is,
    \[ (\tp M)_{-,x} \eqdef M_{x,-} \]
\end{dfn}

\begin{dfn} [\tc{Subdistribution}]
    Let \( X \) be a set. 
    A {\bf subdistribution on \( X \)} is a finitely supported function \( p \colon X \to [0, 1] \) such that 
	\[ S_p \eqdef \sum_{x \in X} p(x) \le 1. \]
    We say that a subdistribution \( p \) is a {\bf distribution} if \( S_p = 1 \).
\end{dfn}

\begin{dfn} [\tc{Subdistribution matrix}] 
    Let \( X, Y \) be sets.
    A {\bf subdistribution matrix of shape \( (X, Y) \)} is a matrix \( M \colon X \times Y \to [0,1] \) such that for all \( x \in X \), the row \( M_{x,-} \) is a subdistribution over Y.
    A subdistribution matrix is called a {\bf distribution matrix} (or stochastic matrix) if all the rows are in fact distributions.
\end{dfn}

\noindent The following result is standard, the proof is left to the reader.
\begin{lem} 
\label{lem:subdistribution_matrices_stable_product}
	Let \( M \) and \( N \) be subdistribution matrices of shape \( (X, Y) \) and \( (Y, Z) \) respectively.
	Then \( MN \) is a subdistribution matrix, which is furthermore a distribution matrix if \( M \) and \( N \) are.
\end{lem}

\begin{dfn}[\tc{Category of subdistribution matrices}] \label{dfn:subdistribution_matrices}
\label{def:subStoc}
	We let $\SubMat$ be the category whose objects sets and morphisms from \( X \) to \( Y \) are the subdistribution matrices of shape \( (X, Y) \). 
	Composition is given by matrix multiplication
\end{dfn}

\begin{rmk}
	By Lemma \ref{lem:subdistribution_matrices_stable_product} and Remark \ref{rmk:product_matrix_unital_assoc}, \( \SubMat \) is a well-defined category.
\end{rmk}

\begin{rmk}
	The category of subdistribution matrices is discussed in the context of discrete partial Markov categories in \cite{DiM23}. \end{rmk}

\begin{dfn} [\tc{Unit distribution}]
    Let \( X \) be a set and \( x \in X \). 
    The {\bf unit distribution on \( x \)}, written \( \dirac{x} \colon X \to [0,1] \), is the distribution defined by
    \begin{equation*}
         \dirac {x} (x') \eqdef
        \begin{cases}
            1 & \text{if } x = x', \\
            0 & \text{otherwise.}
        \end{cases}
    \end{equation*}
   We say that distribution $p\colon X \to [0,1]$ is a unit distribution if it is equal to \( \dirac{x} \) for some $x \in X$.
\end{dfn}

\begin{rmk}
	An unit distribution is sometimes called a Dirac distribution.
\end{rmk}

\begin{dfn} [\tc{Kronecker product}]
    Let \( M, N \) be subdistribution matrices of shape \( (X, Y) \) and \( (U, V) \) respectively. 
    The {\bf Kronecker product of \( M \) and \( N \)}, written \( M \kro N \) is the matrix of shape \( (X \times U, Y \times V) \) defined by
    \begin{equation*}
		((x, u), (y, v)) \mapsto M_{x,y} N_{u,v}.
    \end{equation*}
\end{dfn}

\noindent By Definition \ref{dfn:Kleisli_monoidal_distribution} and Proposition \ref{prop:Kleisli_is_matrix}, the category of subdistribution matrices is a symmetric monoidal category with tensor product given by the Kronecker product and monoidal unit given by the unique distribution matrix of shape \( (\pt, \pt) \). 

\begin{dfn} [\tc{Copy and discard}] \label{dnf:copy_discard_sustoch}
    Let \( X \) be a set.
    The {\bf copy subdistribution matrix on \( X \)} is the subdistribution matrix \( \copyy X \) of shape \( (X, X \times X) \) whose \( x \)\nbd th row is given by
    \begin{equation*}
        (x_1, x_2) \mapsto
        \begin{cases*}
            1 & if $ x_1 = x = x_2$ \\
            0 & otherwise
          \end{cases*}. 
    \end{equation*}
    The {\bf discard subdistribution matrix on \( X \)} is the subdistribution matrix \( \disc X \) of shape \( (X, \pt) \) such that for all $x \in X$, $(\disc X)_{x,*} = 1$.
\end{dfn}

\noindent By Corollary \ref{cor:submat_pmat_copy_discard}, the category \( \SubMat \) is a CDU category with the copy-discard structure defined as above. 
The following results are standard; the reader can reconstruct the proof by unwinding the definitions.
Accordingly, if \( M \) is a subdistribution matrix of shape \( (X, Y) \), we denote by \( \dom{M} \) its domain of definition ((Definition \ref{dfn:domain_def})), which is of shape \( (X, X) \).

\begin{lem} \label{lem:dom_def_subdistribution_matrix}
    Let \( M \) be a subdistribution matrix of shape \( (X, Y) \).
    Then \( \dom{M}\) is the subdistribution matrix given by
    \begin{equation*}
        \dom{M}_{x, x'} = \dirac{x}(x') \sum_{Y} M_{x, -}.
    \end{equation*}
\end{lem}

\begin{lem} \label{lem:total_iff_distribution}
	Let \( M \) be a subdistribution matrix.
	Then \( M \) is a distribution matrix if and only if it is total. 
\end{lem}

\begin{lem} \label{lem:deterministic_iff_quasidirac}
	Let \( M \) be a subdistribution matrix of shape \( (X, Y) \).
    The following are equivalent.
    \begin{enumerate}
        \item \( M \) is deterministic;
        \item for all \( x \in \supp M \), \( M_{x, -} \colon Y \to [0,1] \) is a unit distribution.
    \end{enumerate}
\end{lem}

\subsection{Partial reversibility of subdistribution matrices}

\begin{dfn}[\tc{Subpermutation matrix}] 
	We say that a subdistribution matrix \( M \) of shape \( (X, Y) \) is a {\bf subpermutation matrix} if 
	\begin{enumerate}
        \item for all \( x \in \supp M \), the row of \( M_{x,-} \) is a unit distribution; 
        \item for all \( x, x' \in \supp M \), \( M_{x, -} = M_{x', -} \) if and only if \( x = x' \).
	\end{enumerate}
\end{dfn}

\begin{rmk} \label{rmk:subpermutation_prop}
	By construction, all subpermutation matrices are subdistribution matrices. 
	Furthermore, by Lemma \ref{lem:deterministic_iff_quasidirac}, the first condition is equivalent to saying that \( M \) is deterministic. 
\end{rmk}

\begin{lem} \label{lem:transpose_subpermutation_is_subpermutation}
	Let \( M \) be a subpermutation matrix. 
    Then \( \tp M \) is a subpermutation matrix.
\end{lem}
\begin{proof}
    Let \( (X, Y) \) be the shape of \( M \).
    Fix \( y \in \supp \tp{M} \), and define \( I \) to be the subset of \( X \) such that \( (\tp M)_{y, x} > 0 \).
    Let \( x, x' \in I \), then \( x \) and \( x' \) belong to \( \supp M \), thus \( M_{x, -} \) and \( M_{x', -} \) are unit distribution, necessarily on \( y \), thus \( M_{x, -} = M_{x', -} \).
	By assumption on \( M \), we have \( x = x' \).
    Since \( y \in \supp \tp M \), \( I \) is non-empty, and thus is a singleton \( \set{x} \), hence \( (\tp M)_{y, -} = \dirac{x} \).
	This concludes the proof.
\end{proof}

\begin{prop} \label{prop:partial_inverse_iff_transpose_inverse}
    Let \( M \) be a subdistribution matrix.
    The following are equivalent.
    \begin{enumerate}
        \item \( M \) is a subpermutation matrix; 
        \item \( M \) is a partial isomorphism. 
		In that case, a partial inverse of \( M \) is given by \( \tp M \).
    \end{enumerate}
\end{prop}
\begin{proof} 
    Let \( (X, Y) \) be the shape of \( M \), and suppose first that \( M \) is a subpermutation matrix. 
    For all \( x, x' \in X \), we have by Lemma \ref{lem:dom_def_subdistribution_matrix},
    \begin{equation*}
        \dom{M}_{x, x'} = 
        \begin{cases}
            1 & \text{if } x = x' \in \supp M, \\
            0 & \text{otherwise}.
        \end{cases}
    \end{equation*}
    Then for all \( x, x' \in X \), we compute
    \begin{equation*}
        (M\tp M)_{x, x'} = \sum_{y \in Y} M_{x, y} (\tp M)_{y, x'} = \sum_{y \in Y} M_{x, y} M_{x', y}.
    \end{equation*}
    Suppose first that \( x = x' \) and \( x \in \supp M \), then \( M_{x, -} \) is a unit distribution, hence \( \sum_{y \in Y} M_{x, y} M_{x', y} = 1 \).
    Else, if \( x \) or \( x' \) is not in \( \supp M \), then for all \( y \in Y \), we have \( M_{x, y} M_{x', y} = 0 \), thus \( (M\tp M)_{x, x'} = 0 \).
    Finally, if \( x, x' \in \supp M \), and \( x \neq x' \), then \( M_{x, -} = \dirac{z} \) and \( M_{x', -} = \dirac{z'} \) with \( z \neq z' \), thus for all \( y \in Y \), we still have \( M_{x, y} M_{x', y} = 0 \), proving that \( (M\tp M)_{x, x'} = 0 \).
    This shows that for all subpermutation matrices \( N \), we have \( N\tp N = \dom N \).
    Applied to \( \tp M \), which is a subpermutation matrix by Lemma \ref{lem:transpose_subpermutation_is_subpermutation}, we also get \( \tp M M = \tp M \tp {(\tp M)} = \dom {\tp M} \), proving that \( \tp M \) is a partial inverse of \( M \).
    
    Conversely, suppose that \( M \) has a partial inverse \( N \).
    We fix \( x \in \supp M \), and we show that \( M_{x, -} \) is a unit distribution. 
	By Lemma \ref{lem:dom_def_subdistribution_matrix}, 
	\begin{equation*}
		\overline{M}_{x,x} = \delta_{x}(x)  \sum_{y \in Y} M_{x,y} = \sum_{y \in Y} M_{x,y}.
	\end{equation*}
    Thus,
    \begin{equation*}
		\dom{M}_{x, x} - (MN)_{x, x} =  \sum_{y \in Y} M_{x,y} - \sum_{y \in Y} M_{x,y} N_{y,x}  = \sum_{y \in Y} M_{x, y}(1 - N_{y, x}) = 0.
    \end{equation*}
    This implies that for all \( y \in Y \), \( M_{x, y} = 0 \) or \( N_{y, x} = 1 \).
    Since \( x \in \supp M \), there exists \( y_0 \in Y \) such that \( M_{x, y_0} > 0 \), which implies that \( N_{y_0, x} = 1 \).
    Since \( N \) is a subdistribution matrix, necessarily \( N_{y_0, -} = \dirac{x} \).
    Now, \( (NM)_{y_0, y_0} = \dom{N}_{y_0, y_0} \), but
    \begin{align*}
        (NM)_{y_0, y_0} &= \sum_{x' \in X} N_{y_0x'} M_{x', y_0} \\ 
        				&=  \dirac{x}(x') \sum_{x' \in X} M_{x', y_0} \\ 
        				&= M_{x, y_0}.
    \end{align*}
    By Lemma \ref{lem:dom_def_subdistribution_matrix}, \( \dom{N}_{y_0, y_0} = 1 \). 
    Therefore, \( M_{x, y_0} = 1 \), and since \( M \) is a subdistribution matrix, necessarily \( M_{x, -} = \dirac{y_0} \). Thus, for all $x \in \supp M$, $M_{x,-}$ is a unit distribution.
    
    Finally, suppose that \( M_{x, -} = M_{x', -} \) for \( x, x' \in \supp M \).
	Consider the distribution \( p \eqdef \frac{1}{2}(\dirac{x} + \dirac{x'}) \), and call \( S \eqdef \sum_{Y} M_{x, -} \).
    Then on the one hand, we have by Lemma \ref{lem:dom_def_subdistribution_matrix}, 
    \begin{align*}
        (pMN)_x &= \sum_{a \in X} p_{*,a} (MN)_{a, x} = \sum_{a \in X} p_{a} (MN)_{a, x} = 
        \sum_{a \in X} p_{a} \dom{M}_{a, x} \\ 
         &= \left( \frac{1}{2} \dirac{x}(a) \dom{M}_{a,x} + \frac{1}{2} \dirac{x'}(a') \dom{M}_{a,x} \right)   = \left( \frac{1}{2} \dom{M}_{x,x} + \frac{1}{2} \dom{M}_{x',x} \right) \\ 
         &= \left( \frac{1}{2} S + \frac{1}{2} \dirac{x'}(x) S \right)  \\ 
         &= \frac{S}{2}(1 + \dirac{x'}(x)).
    \end{align*}
    On the other hand, for all \( y \in Y \), one computes that 
    \[ (pM)_y = \left( \frac{1}{2}(\dirac{x} + \dirac{x'})M \right)_y = \frac{1}{2} M_{x, y} + \frac{1}{2} M_{x', y} = M_{x,y}, \]
	where in the last step we used that \( M_{x, y} = M_{x', y} \).
    Thus, 
    \begin{align*}
         (pMN)_x &= \sum_{y \in Y} (pM)_y N_{y, x} \\
                 &= \sum_{y \in Y} M_{x, y} N_{y, x} \\
                 &= \dom{M}_{x, x} = S.
    \end{align*}
    Therefore, we must have \( \frac{S}{2}(1 + \dirac{x}(x')) = S \).
    Since \( x \in \supp M \), \( S > 0 \), so we obtain \( \dirac{x'}(x) = 1 \), that is, \( x = x' \).
    This shows that \( M \) is a subpermutation matrix, and concludes the proof.
\end{proof}

\begin{cor} \label{cor:subpermutation_stable_product_Kronecker}
    Let \( M, N \) be subpermutation matrices.
    Then
    \begin{enumerate}
        \item \( MN \), if defined, is a subpermutation matrix;
        \item \( M \kro N \) is a subpermutation matrix.
    \end{enumerate}
	Consequently, subpermutation matrices form a monoidal subcategory of \( \SubMat \).
\end{cor}
\begin{proof}
    By Proposition \ref{prop:partial_inverse_iff_transpose_inverse}, subpermutation matrices are, in the CDU category $\SubMat$, deterministic partial isomorphisms with deterministic partial inverses. 
    The statements then follow from Lemma \ref{lem:partial_iso_compose} and Lemma \ref{lem:partial_iso_tensor} respectively.
\end{proof}

\begin{rmk}
	The subcategory of subpermutation matrices does not inherit the CDU structure of $\SubMat$: for all sets $X$, the copy matrix \( \copyy X \) is a subpermutation matrix, however, the discard \( \disc X \) is not a subpermutation matrix whenever $X$ has two or more elements. 
\end{rmk}	

\noindent 

\subsection{Partitioned matrices}

As a final step in developing a model for generalized reversible computing, we consider subdistribution matrices between partitioned sets. 

\begin{dfn} [\tc{Equivalent subdistribution}] \label{defn:equivalent-subdistribution}
	\label{dfn:equivalent_subdistribution}
    Let \( (X, \sim) \) be a partitioned set, \( p, q \) be two subdistributions on \( X \).
    We say that \( p \) and \( q \) are {\bf equivalent subdistributions}, and write \( p \sim q \) if for all \( x \in X \), 
	\begin{equation*}
		\sum_{z \in \eqclass{x}} p_z = \sum_{z \in \eqclass{x}} q_z.
	\end{equation*}
\end{dfn}

\begin{dfn}[\tc{Partitioned matrices}] \label{defn:partitioned-matrices}
\label{Defn:partitioned-matrix}
	Let $(X, \sim_X), (Y, \sim_Y)$ be partitioned sets. 
	We say a matrix $M$ of shape \( (X, Y) \) is a {\bf partitioned matrix (with respect to \( \sim_X \) and \( \sim_Y \))} if and only if for all subdistributions $p, q\colon X \to [0,1]$ such that $p \sim q$, we have
    \begin{equation*}
        pM \sim qM.
    \end{equation*}
\end{dfn}

\begin{rmk}
	We will often leave the equivalence relation on \( X \) and \( Y \) implicit, and just say that \( M \) is partitioned matrix of shape \( (X, Y) \).
\end{rmk}

\begin{dfn}[\tc{Equivalent partitioned matrices}] \label{defn:equivalent-partitioned-matrices}
	Let $M, N$ be partitioned matrices of shape \( (X, Y) \). 
	We say that {\bf $M$ and $N$ are equivalent}, and write \( M \sim N \), if and only if for all subdistributions \( p \colon X \to [0, 1] \), we have \( pM \sim pN \).
\end{dfn}
 
\noindent We shall now give a more explicit characterisation of partitioned matrices. 

\begin{prop}\label{prop:partitioned-matrices-concretely}
Let \( (X, \sim) \) and \( (Y, \sim) \) be partitioned sets, and $M$ be a subdistribution matrix of shape \( (X, Y) \).
The following are equivalent:
\begin{enumerate}
	\item $M$ is a partitioned matrix;
	\item for all \( x, x' \in X \) such that \( x \sim x' \), we have \( \dirac{x} M \sim \dirac{x'} M \);
	\item for all \( x, x' \) in \( X \) such that \( x \sim x' \), and for all $y$, 
    \[ \sum_{z \in \eqclass{y}} M_{x, z} = \sum_{z \in \eqclass{y}} M_{x', z}. \]
\end{enumerate}
\end{prop}
\begin{proof}
	Notice that if \( x \in X \), \( \dirac{x}M = M_{x, -} \).
	Thus the last two conditions are equivalent by definition.
	Suppose first that \( M \) is partitioned.
	Then the second condition follows immediately, since if \( x \sim x' \), also \( \dirac{x} \sim \dirac{x'} \). 
	Conversely, suppose that the last condition holds, consider two distributions \( p, q \) on \( X \) such that \( p \sim q \), and let \( y \in Y \).
	Given \( x \in X \), we write \( M_{x, \eqclass{y}} \eqdef \sum_{b \in \eqclass{y}} M_{x, b} \).
	Notice that if \( x \sim x' \), \(  M_{x, \eqclass{y}} =  M_{x', \eqclass{y}} \).
	We compute
	\begin{align*}
		\sum_{b \in \eqclass{y}} (pM)_b &= \sum_{b \in \eqclass{y}} \sum_{x \in X} p_x M_{x, b} \\
		&= \sum_{x \in X} p_x  M_{x, \eqclass{y}} \\
		&= \sum_{a \in X/\sim} \left( \sum_{x \in \eqclass{a}} p_x M_{x, \eqclass{y}} \right) \\
		&= \sum_{a \in X/\sim} \left( \sum_{x \in \eqclass{a}} p_x M_{a, \eqclass{y}} \right) \\ 
		&= \sum_{a \in X/\sim} M_{a, \eqclass{y}} \left( \sum_{x \in \eqclass{a}} p_x \right) \\ 
		&= \sum_{a \in X/\sim} M_{a, \eqclass{y}} \left( \sum_{x \in \eqclass{a}} q_x \right),
	\end{align*}
	where the last step holds since \( p \sim q \).
	Proceeding to a similar computation, we find that the last term is equal to \( \sum_{b \in \eqclass{y}} (qM)_b \) and conclude the proof.
\end{proof}

\begin{dfn} [\tc{Category of partitioned matrices}] \label{dfn:partionned_matrices} \label{defn:category-partitioned-matrices}
	We let \( \PSubMat \) be the category where objects are partitioned sets, and morphisms from \( (X, \sim) \) to \( (Y, \sim) \) are partitioned matrices of shape \( (X, Y) \).
	Composition is given by matrix multiplication. 
\end{dfn}

\begin{rmk}
    By Proposition \ref{prop:Kleisli_is_matrix}, the category \( \PSubMat \) is well defined, and in particular, the multiplication of two partitioned matrix is again a partitioned matrix.  
	By Corollary \ref{cor:submat_pmat_copy_discard}, \( \PSubMat \) is also a CDU category with monoidal product given by the Kronecker product, and copy-discard structure given by Definition \ref{dnf:copy_discard_sustoch}. 
\end{rmk}
 
\subsection{Relating partitioned matrices and subdistribution matrices} 
\label{subsec:Aggregation-functor}
 
\noindent We now examine a crucial functor between $\PSubMat$ and $\SubMat$.

\begin{dfn}[\tc{Aggregation functor}] \label{dfn:aggregation}
	Let \( (X, \sim) \), \( (Y, \sim) \) be partitioned sets, and \( M \) be a partitioned matrix of shape \( (X, Y) \).
	The {\bf aggregation matrix} is the matrix \( \Q M \) of shape \( (\quot{X}{\sim}, \quot{Y}{\sim}) \) defined by
	\begin{equation*}
		(\eqclass{x}, \eqclass{y}) \mapsto \sum_{z \in [y]} M_{x, z}.
	\end{equation*}
	By Corollary \ref{cor:aggregation_functor}, this defines a strong monoidal functor
	\begin{equation*}
		\Q \colon (\PSubMat, \kro, \pt) \to (\SubMat, \kro, \pt),
	\end{equation*}
	called the {\bf aggregation functor}.
\end{dfn}

\begin{dfn} [\tc{Locally finite partitioned set}]
	Let \( (X, \sim) \) be a partitioned set.
	We say that \( X \) is locally finite if for all \( x \in X \), \( \eqclass{x} \) is finite. 
\end{dfn}

The following lemma formalizes the idea that every computational process between finite computational states is supported by a physical process.

\begin{lem}\label{lem: Q_is_full}
	The functor $\Q\colon \PSubMat \to \SubMat$ is essentially surjective.
	Furthermore, for all partitioned sets \( (X, \sim) \) and \( (Y, \sim) \) such that \( (Y, \sim) \) is locally finite, the induced function of hom-sets
	\begin{equation*}
		\Q\colon \PSubMat((X, \sim), (Y, \sim)) \to \SubMat(\quot{X}{\sim}, \quot{Y}{\sim})
	\end{equation*}
	is surjective.
\end{lem}
\begin{proof}
	Consider any object $X \in \SubMat$, then $ X = \Q((X,=)) $.
	Thus, \( \Q \) is essentially surjective.
	Let $(X,\sim)$ and $(Y, \sim)$ be partitioned sets with \( (Y, \sim) \) locally finite, and $M$ be a partitioned matrix of shape \( (\Q (X, \sim), \Q (Y, \sim)) = (\quot{X}{\sim}, \quot{Y}{\sim}) \).
	For all \( y \in Y \), define \( c_y \) to be the cardinality of \( \eqclass{y} \), which is finite by assumption, and strictly positive since \( y \in \eqclass{y} \). 
	Define the matrix \( N \) of shape \( (X, Y) \) by letting 
	\begin{equation*}
		 N_{x,y} \eqdef \frac{ M_{\eqclass{x}, \eqclass{y}}}{c_y},
	\end{equation*}
	for all \( x \in X \) and \( y \in Y \).
	We show first that $N$ is a subdistribution matrix.
	Let $x \in X$, then
	\begin{align*}
	\sum_{y \in Y} N_{x,y} & \eqdef \sum_{y \in Y} \frac{ M_{\eqclass{x}, \eqclass{y}}}{c_y} \\  
	&= \sum_{\eqclass{y'} \in \quot{Y}{\sim}}  ~~ \sum_{y \in \eqclass{y'}} \frac{ M_{\eqclass{x}, \eqclass{y}}}{c_y} \\ 
	& = \sum_{\eqclass{y'} \in \quot{Y}{\sim}}  ~~ \sum_{y \in \eqclass{y'}} \frac{ M_{\eqclass{x}, \eqclass{y'}}}{c_{y'}} \\ 
	&= \sum_{\eqclass{y'} \in \quot{Y}{\sim}} M_{\eqclass{x}, \eqclass{y'}} \frac{ \sum_{y \in \eqclass{y'}} 1 }{c_{y'}} \\ 
	& = \sum_{\eqclass{y'} \in \quot{Y}{\sim}} M_{\eqclass{x}, \eqclass{y'}} \leq 1.
	\end{align*}
	This proves that \( N \) is a subdistribution matrix.
	Next, we check that $\Q N = M$.
	Let $x \in X$ and $y \in Y$, then
	\[ (\Q N)_{\eqclass{x}, \eqclass{y}} = \sum_{y' \in \eqclass{y}} N_{x,y'} = \sum_{y' \in \eqclass{y}} \frac{ M_{\eqclass{x}, \eqclass{y'}}}{c_{y'}} = 
	\sum_{y' \in \eqclass{y}} \frac{M_{\eqclass{x}, \eqclass{y}}}{c_y} = M_{\eqclass{x}\eqclass{y}}. \]
	Finally, we show that \( N \) is partitioned.
	By Proposition \ref{prop:partitioned-matrices-concretely}, this is to say that for \( x, x' \in X \) such that $x \sim x' \in X$, and $y \in Y$, we have
	\( (\Q N)_{\eqclass{x}, \eqclass{y}} = (\Q N)_{\eqclass{x'}, \eqclass{y}} \).
	Since \( \Q N = M \), this equality follows from the fact that \( M \) is partitioned.
\end{proof}

\section{Fundamental theorem of generalized reversible computing}
\label{sec:fundamental-theorem}

The purpose of this section is to state the conditions for the equivalence of physical and computational reversibility utilizing the compositional frameworks developed in the previous section: 
\begin{itemize}
	\item The symmetric monoidal category of partitioned matrices \( \PSubMat \), see Definition \ref{Defn:partitioned-matrix};
	\item The symmetric monoidal category of subdistribution matrices,  \( \SubMat \), see Definition \ref{dfn:subdistribution_matrices};
	\item The aggregation functor \( \Q \colon (\PSubMat, \kro, \pt) \to (\SubMat, \kro, \pt) \), see Definition \ref{dfn:aggregation}. 
\end{itemize}  

The main innovation in the set-theoretic framework of generalized reversible computing of \cite{Fra18} is accounting for the statistical context of a computation system, that is, the probability distribution over the set of states. This allows tracking changes in entropy (at physical and computational levels) as computations proceed, and defining what it means for a computational step to be reversible at different levels.   

\subsection{Conservation of entropy}

\begin{dfn} [\tc{Entropy}] \label{dfn:entropy}
    Let \( p \colon X \to [0, 1] \) be a subdistribution. 
    The {\bf entropy of \( p \)}, written $\H(p)$, is the positive real number defined by 
    \begin{equation}
    	\label{eqn: entropy}
        \H(p) \eqdef - \sum_{x \in X} p_x \log p_x.
    \end{equation}
\end{dfn}

\begin{rmk}
    The function \( f \colon t \mapsto t \log t \) defined on \( (0, 1] \) can be uniquely extended into a continuous function \( \bar f \) defined on \( [0, 1] \) by letting \( \bar f(0) \eqdef 0 \).
    We will always use this unique extension \( \bar f \) throughout the article. 
    Thus, the entropy of the zero subdistribution is well defined and equal to \( 0 \).
\end{rmk}

\begin{rmk} \label{rmk:log_superadditive}
   The function \( t \mapsto t \log t \) is superadditive. This means that for all positive real numbers \( s, t \), we have
   \begin{equation*}
    s \log s + t \log t \le (s + t) \log (s + t),
   \end{equation*}
   the inequality being strict if \( s, t \notin \set{0, 1} \).
\end{rmk}

\begin{lem}
\label{Lem: Shannon entropy is additive}
	Let \( p, q \) be subdistributions. 
	Then, $\H( p \kro q) \leq \H(p) + \H(q)$,
	with equality if and only if $p$ and $q$ are distributions.
\end{lem}
\begin{proof}
	Let \( X, Y \) be sets such that \( p \colon X \to [0, 1] \) and \( q \colon Y \to [0, 1] \), and consider $(x,y) \in X \times Y$. 
	By definition, $(p \kro q)_{(x,y)} = p_x q_y$.
	Then
	\begin{align*}
		\H(p \kro q) &= - \sum_{(x,y)} p_x q_y \log(p_x q_y) \\
		&= - \sum_{(x,y)} p_x q_y (\log p_x + \log q_y) \\
		&=  - \sum_{(x,y)} p_x  q_y \log p_x - \sum_{(x,y)} p_x  q_y \log q_y \\
		&=  - \sum_{x} p_x \log p_x \sum_y q_y - \sum_{x} p_x \sum_y q_y \log q_y \\
		&=   \H(p) \left(\sum_y q_y\right)  + \H(q)\left(\sum_{x} p_x\right) \\
		& \leq \H(p) + \H(q),
	\end{align*}
	the latter being an equality if and only if \( \sum_{x} p_x = 1 \) and \( \sum_y q_y = 1 \), that is, if and only if \( p \) and \( q \) are distributions.	
\end{proof}

\begin{lem} \label{lem:kernel_entropy}
    Let \( p \colon X \to [0, 1] \) be a subdistribution such that \( \H(p) = 0 \).
    Then either \( p \) is the zero subdistribution or a unit distribution. 
\end{lem}
\begin{proof}
    We have \( \H(p) = \sum_{x \in X} - p_x \log p_x = 0 \).
    Since \( - t \log t \geq 0 \) for all \( t \in [0, 1] \), this implies that for all \( x \in X \), \( p_x \log p_x = 0 \), i.e. that \( p_x \in \set{0, 1} \).
    Thus, either \( p \) is the zero subdistribution, or there exists \( x \in X \), necessarily unique, such that \( p_x = 1 \), that is, \( p = \dirac{x} \).
\end{proof}

\begin{lem} \label{lem:deterministic_iff_entropy_decreasing}
    Let $M$ be a subdistribution matrix of shape \( (X, Y) \).
    The following are equivalent.
    \begin{enumerate}
        \item \( M \) is deterministic;
        \item for all subdistributions \( p \) on \( X \), \( \H(p M) \le \H(p) \).
    \end{enumerate}
\end{lem} 
\begin{proof}
    Suppose that \( M \) is deterministic, then by Lemma \ref{lem:deterministic_iff_quasidirac} there is a function \( \fun{m} \colon \supp M \to Y \) such that for all \( x \in \supp M \), \( M_{x, -} = \dirac{\fun{m}(x)} \). 
    Therefore, for all \( y \in Y \), we have 
    \begin{equation*}
        (pM)_y = \sum_{x \in X} p_x M_{x, y} = \sum_{x \in X: \fun{m}(x) = y} p_x.
    \end{equation*}
    We compute
    \begin{align*}
    \H(p M) &= - \sum_{y \in Y} (pM)_y \log (pM)_y \\
	&= - \sum_{y \in Y} \left(\sum_{x \in X: \fun{m}(x) = y} p_x\right) \log\left( \sum_{x \in X: \fun{m}(x) = y} p_x \right) \\
                &\le - \sum_{y \in Y} \sum_{x \in X: \fun{m}(x) = y} p_x \log p_x & (t \mapsto t \log t) \text{ superadditive } \\
                &= - \sum_{x \in X} p_x \log p_x = \H(p).
    \end{align*} 
    Conversely, let \( x \in X \), then \( \H(\dirac{x} M) \le \H(\dirac{x}) = 0 \), so by Lemma \ref{lem:kernel_entropy}, \( \dirac{x} M = M_{x, -} \) is either the zero distribution or a unit distribution.
    By Lemma \ref{lem:deterministic_iff_quasidirac}, \( M \) is deterministic.
    This concludes the proof.
\end{proof}

\begin{prop}\label{prop:partial_inv_iff_deterministic_entropy_preserving}
    Let \( M \) be a subdistribution matrix of shape \( (X, Y) \).
    The following are equivalent.
    \begin{enumerate}
        \item \( M \) is a subpermutation matrix;
        \item for all subdistributions \( p \) on \( X \) such that \( \supp p \subseteq \supp M \), \( \H(pM) = \H(p) \).
    \end{enumerate}
\end{prop}
\begin{proof}
    Suppose that \( M \) is a subpermutation matrix, and let \( p \) be a distribution on \( X \) with the same support as \( M \). 
    Since \( M \) is a subpermutation matrix, by Proposition \ref{prop:partial_inverse_iff_transpose_inverse}, it has a partial inverse given \( \tp{M} \), which is also a subpermutation matrix.
    In particular, Lemma \ref{lem:deterministic_iff_quasidirac} shows that both \( M \) and \( \tp M \) are deterministic.
    Therefore, applying Lemma \ref{lem:deterministic_iff_entropy_decreasing} to \( \tp M \), we have
    \begin{equation*}
        \H(p \dom{M}) = \H(p M\tp{M}) \le \H(pM).
    \end{equation*}
    Now by Lemma \ref{lem:dom_def_subdistribution_matrix}, 
    \begin{equation*}
        (p \dom{M})_x = p_x \sum_{y \in Y} M_{x, y} = 
        \begin{cases}
            p_x & \text{if } x \in \supp M \\
            0 & \text{else.}
        \end{cases}
    \end{equation*}
    Since \( \supp p \subseteq \supp M \), we deduce that \( p \dom{M} = p \), hence \( \H(p) \le \H(pM) \).
    Since Lemma \ref{lem:deterministic_iff_entropy_decreasing} also implies \( \H(pM) \le \H(p) \), we deduce the equality.

    Conversely, let \( x \in \supp M \), then \( \supp \dirac{x} = \set{x} \subseteq \supp M \), and by Lemma \ref{lem:kernel_entropy}, \( \H(\dirac{x}) = 0 \). 
	By assumption, we have that \( \H(\dirac{x}) = \H(\dirac{x} M) = \H(M_{x, -}) \). Therefore, \( \H(M_{x, -}) = 0 \). 
    Since \( x \in \supp M \), \( M_{x, -} \) is not the zero distribution, hence by Lemma \ref{lem:kernel_entropy}, \( M_{x, -} \) is a unit distribution.
    Finally, let \( x, x' \in \supp M \) and suppose that \( M_{x, -} \) and \( M_{x', -} \) are both equal to \( \dirac{y} \) for some \( y \in Y \).
    Let \( p \eqdef \frac{1}{2} \dirac{x} + \frac{1}{2} \dirac{x'} \), we clearly have that \( \supp p \subseteq \supp M \).
    Now \( pM = \dirac{y} \), so \( \H(pM) = 0 \). 
    By assumption, this implies that \( \H(p) = 0 \), so by Lemma \ref{lem:kernel_entropy}, \( p \) is either the zero subdistribution or a unit distribution.
    Since \( p \) is visibly not the zero subdistribution, it has to be a unit distribution, which is possible only if \( x = x' \).
    This shows that \( M \) is a subpermutation matrix and concludes the proof.
\end{proof}

\begin{cor} \label{cor:partial_inverse_iff_subpermutation_iff_entropy_preserving}
    Let \( M \) be a subdistribution matrix of shape \( (X, Y) \).
    The following are equivalent.
    \begin{enumerate}
        \item \( M \) has a partial inverse;
        \item \( M \) is a subpermutation matrix;
        \item for all subdistributions \( p \) on \( X \) such that \( \supp p \subseteq \supp M \), \( \H(pM) = \H(p) \). 
    \end{enumerate}
    In that case, a partial inverse of \( M \) is given by \( \tp{M} \).
\end{cor}
\begin{proof}
    Follows from Proposition \ref{prop:partial_inverse_iff_transpose_inverse} and Proposition \ref{prop:partial_inv_iff_deterministic_entropy_preserving}. 
\end{proof}

\subsection{Categories of physical and computational contexts}
\label{Subsec:Categories-of-physical-and-computational-contexts}

In this section, we construct categories representing statistical contexts (at physical and computational levels) of computational systems. We represent physical contexts via partitioned matrices and computational contexts via subdistribution matrices.  
First, we consider the subcategories of the coslice of \( \PSubMat \) and \( \SubMat \) under the tensor unit \( \pt \).

\begin{dfn}[\tc{Category of computational contexts, $\Comp$}] \label{defn:comp}
	A {\bf computational context} is a distribution \( p \colon X \to [0, 1] \) on a set \( X \).
	Given two computational contexts \( p \colon X \to [0, 1] \) and \( q \colon Y \to [0, 1] \) on sets \( X \) and \( Y \) respectively, a {\bf computational transformation} is a subdistribution matrix \( M \) of shape \( (X, Y) \) such that \( pM = q \).
	We write \( M \colon p \tf q \).
	The {\bf category of computational contexts} is the category \( \Comp \) whose objects are computational contexts, and morphisms are computational transformations.
	Composition is given by matrix multiplication.
\end{dfn}

\begin{rmk}
	Since a distribution \( p \colon X \to [0, 1] \) on a set \( X \) is equivalent to a distribution matrix of shape \( (\pt, X) \), one checks that the category of computational contexts is indeed a subcategory of the coslice category \( \slice{\pt}{\SubMat} \).
\end{rmk}

\begin{dfn}[\tc{Category of physical contexts, $\Phy$}] \label{defn:phy}
	A {\bf physical context} is a distribution \( p \colon X \to [0, 1] \) on partitioned set \( (X, \sim) \).
	Given two physical contexts \( p \colon X \to [0, 1] \) and \( q \colon Y \to [0, 1] \) on the partitioned sets \( (X, \sim) \) and \( (Y, \sim) \) respectively, a {\bf (closed) physical transformation} is a partitioned matrix \( M \) of shape \( (X, Y) \) such that \( pM = q \) and \( \H(p) = \H(q) \).
	We write \( M \colon p \tf q \).
	The {\bf category of physical contexts} is the category \( \Phy \) whose objects are physical contexts, and morphisms are physical transformations.
	Composition is given by matrix multiplication.
\end{dfn}

\begin{comm}
	The condition \( \H(p) = \H(q) \) reflects the fact that \( \Phy \) models closed physical systems.
	By \cite[Theorem 1]{Fra18}, the entropy of such a system is preserved under physically viable transformations.
\end{comm}

\begin{rmk}
	Since a distribution \( p \colon X \to [0, 1] \) on a set \( X \) is equivalently a distribution matrix of shape \( (\pt, X) \), the category of physical contexts is indeed a subcategory of the coslice category \( \slice{\pt}{\PSubMat} \).
\end{rmk}

\begin{rmk}
	By Lemma \ref{lem:physical_tf_tensor}, both the categories \( \slice{\pt}{\SubMat} \) and \( \slice{\pt}{\PSubMat} \) inherit the monoidal structure given by the Kronecker product on \( \SubMat \) and \( \PSubMat \) respectively. 

	Furthermore, the Kronecker product is in fact the cartesian product of \( \Comp \), making it canonically a CDU category in which every morphism is deterministic; however, this CDU-structure has no clear physical interpretation, since it removes the essential non-determinism built-in the theory of subdistribution matrices.

	Finally, the category \( \Phy \) does not have a CDU structure for the Kronecker product, as this would require the discard to preserve entropy, implying that all physical contexts have entropy \( \H(\pt) = 0 \).   
\end{rmk}

\begin{lem}
\label{lem:physical_tf_tensor}
	Let \( M, N \) be physical transformations.
	Then $M \otimes N$ is a physical transformation.
\end{lem}
\begin{proof}
	Suppose that \( M \) and \( N \) are of type \( p \tf q \) and \( p' \tf q' \) respectively, where by definition, \( p, q, p' \) and \( q' \) are distributions.
	Then by Lemma \ref{Lem: Shannon entropy is additive} twice,
	\begin{align*}
		\H((p \kro p')(M \kro N)) &= \H((pM) \kro (p'N)) = \H(pM) + \H(p'N) = \H(q) + \H(q') \\
		&= \H(q \kro q').
	\end{align*} 
	This concludes the proof.
\end{proof}

\noindent By functoriality of the coslice construction, the aggregation functor $\Q$ from Definition \ref{dfn:aggregation} defines again a functor \( \Q \colon \slice{\pt}\PSubMat \to \slice{\pt}\SubMat \). 
It is straightforward to check that this restricts to a functor
\begin{equation*}
	\Q \colon \Phy \to \Comp.
\end{equation*}
In this perspective, \( \Q \) associates to given a physical context $p \colon X \to [0, 1]$ of the state space $(X,\sim)$, its associated computational context $\Q(p)$ of the computational state \( \quot{X}{\sim} \).

Since physical systems are finite, we may apply Lemma \ref{lem: Q_is_full}.
The fact that $\Q \colon \pt/\PSubMat \to \pt/\SubMat$ is essentially surjective indicates that every computational context has a corresponding physical context. 
The fact that $\Q$ is full indicates that every computational transformation can be modelled by a physical operation.

\subsection{Physical reversibility via non-computational entropy}
\label{Subsec:physical-reversibility}

In this section, we use non-computational entropy \cite[Definition 7]{Fra18} as a parameter to describe reversibility of a computation at the physical level. Non-computational entropy represents the entropy at the physical level when the information at the computational level is known. 

\begin{dfn}[\tc{Physical and computational entropies}] \label{defn:entropies}
\label{defn:entropies}
	Let \( p \) be a physical context.
	\begin{itemize}
		\item The {\bf physical entropy} of \( p \) is the quantity 
		\begin{equation*}
			\Hphy(p) \eqdef \H(p).
		\end{equation*}
		\item The {\bf computational entropy} of \( p \) is the quantity
		\begin{equation*}
			\Hcomp(p) \eqdef \H(\Q (p)).
		\end{equation*}
		\item The {\bf non-computational entropy} of \( p \) is the quantity
		\begin{equation*}
			 \Hnc(p) \eqdef \Hphy(p) - \Hcomp(p).	
		\end{equation*}
	\end{itemize}
\end{dfn}

\noindent The non-computational entropy is indeed equal to the physical entropy conditioned on the computational entropy, that is, 
\begin{equation*}
	\Hnc(p) = \H(\Hphy(p) \mid \Hcomp(p)),
\end{equation*}
see \cite[Theorem 2]{Fra18} for more detail.

\begin{dfn} \label{def: non-entropy-ejecting} [\tc{Non-entropy-ejecting processes}]
	Let \( M \colon p \tf q \) be a physical transformation.
	We say that \( M \) is {\bf entropy-ejecting} if $M$ causes an non-zero increase in the non-computational entropy, that is, if 
	\begin{equation*}
		\Hnc(q) > \Hnc(p).
	\end{equation*}
	Thus, \( M \) is non-entropy-ejecting if \( \Hnc(q) \le \Hnc(p) \)
	\end{dfn}

\noindent Let \( \bbR \) be the chaotic\footnote{This means that there is exactly one morphism per any pair of objects} category on the real line, endowed with its symmetric monoidal structure given by the addition.

\begin{lem} \label{lem:entropies_monoidal}
	The physical, computational, and non-computational entropies define strict monoidal functors 
	\begin{equation*}
		\Hphy, \Hcomp, \Hnc \colon (\Phy, \kro, \pt) \to (\bbR, +, 0).
	\end{equation*}
\end{lem}
\begin{proof}
	These construction are evidently functorial by sending \( M \colon p \tf q \) to the unique morphism from \( \H(p) \) to \( \H(q) \).
	By Lemma \ref{Lem: Shannon entropy is additive} and Corollary \ref{cor:aggregation_functor}, \( \Hphy \) and \( \Hcomp \) are monoidal.
	It follows that \( \Hnc \) is monoidal.
	This concludes the proof.
\end{proof}

Next, we check that non-entropy ejecting processes compose in sequence and in parallel.

\begin{lem}
\label{lem:n.e.e_compose_tensor}
	Let \( M \) and \( N \) be non-entropy-ejecting transformations.
	Then
	\begin{enumerate}
		\item \( MN \), if defined, is non-entropy-ejecting;
		\item \( M \kro N \) is non-entropy-ejecting.
	\end{enumerate}
\end{lem}
\begin{proof}
	Suppose that \( MN \colon p \tf r \) is defined.
	Then 
	\begin{equation*}
		\Hnc(r) - \Hnc(p) = \underbrace{(\Hnc(r) - \Hnc(pM))}_{> 0} - \underbrace{(\Hnc(p) - \Hnc(pM))}_{< 0} > 0.
	\end{equation*}
	Thus \( MN \) is non-entropy-ejecting.
	Next, let \( M \colon p \tf q \) and \( N \colon p' \tf q' \) be the type of \( M \) and \( N \).
	Then, by Lemma \ref{lem:entropies_monoidal}, 
	\begin{equation*}
		\Hnc(q \kro q') = \Hnc(q) + \Hnc(q') \leq \Hnc(p) + \Hnc(p') = \Hnc(p \kro p').
	\end{equation*}
	This means that \( M \kro N \colon p \kro p' \tf q \kro q' \) is non-entropy-ejecting.
\end{proof}

Non-entropy ejection captures the notion of reversibility for physical transformations. 
We find that physical transformations of a computational system to be reversible precisely when the transformation does not lose computational entropy.

\begin{lem} \label{Lem: n.e.e._reversible}
	Let \( M \colon p \tf q \) be a physical transformation.
	Then \( M \) is non-entropy-ejecting if and only if \( \Hcomp(p) \le \Hcomp(q) \).
\end{lem}
\begin{proof}
	By definition,
	\begin{align*}
		\Hnc(q) - \Hnc(p) &= \Hphy(q) - \Hcomp(q) - \Hphy(p) + \Hcomp(p) \\ 
		&= \Hcomp(p) - \Hcomp(q).
	\end{align*} 
	Thus \( \Hcomp(p) \le \Hcomp(q) \) if and only if \( \Hnc(q) \le \Hnc(p) \) if and only if \( M \) is non-entropy-ejecting.
\end{proof}

\subsection{Computational reversibility}
\label{Subsec:computational-rev}

In generalized reversible computing, every computation is considered within a specific statistical operating context. 
Considering reversibility of computations with respect to specific context leads to the notion of conditionally reversible operation, which considers reversibility over a subset of states. 
Consequently, we introduce the following notation.
Given matrix \( M \) of shape \( (X, Y) \) and \( p \) a subdistribution on \( X \), we write \( \restr{M}{p} \) for the matrix of shape \( ((\supp p) \cap X, Y) \) defined by \( (x, y) \mapsto M_{x, y} \).

\begin{dfn}[\tc{Conditional reversibility}] \label{dfn:condrev}
	Let $M \colon p \tf q$ be a deterministic computational transformation. 
	We say that \( M \) is {\bf conditionally reversible} if 
	\begin{enumerate}
		\item \( \supp p \subseteq \supp M \), and
		\item the matrix $\restr{M}{p}$ has a partial inverse. 
	\end{enumerate}
\end{dfn}

\begin{rmk}
	If the first condition holds, the second condition is equivalent to asking, by Proposition \ref{prop:partial_inverse_iff_transpose_inverse}, that \( \restr{M}{p} \) is a subpermutation matrix.
\end{rmk}

Conditionally reversible processes compose in sequence and in parallel.

\begin{lem}
\label{lem:condrev_maps_compose}
	Let \( M \colon p \tf q \) and \( N \colon q \tf r \) be conditionally reversible computational transformations.
	Then \( MN \) is conditionally reversible.
\end{lem}
\begin{proof}
	First of all, \( MN \) is deterministic by Lemma \ref{lem:det_tot_composition}.
	Suppose that \( p, q \), and \( r \) are distributions over the sets \( X, Y \) and \( Z \), respectively. 
	Let \( M' \) be the matrix of shape \( (\supp p, \supp q) \) defined by \( (x, y) \mapsto M_{x, y} \).
	We claim that \( M' \) has a partial inverse. 
	Indeed, \( \restr{M}{p} \) has a partial inverse, so by Corollary \ref{cor:partial_inverse_iff_subpermutation_iff_entropy_preserving}, it is a subpermutation matrix.
	Therefore, \( M' \) is also a subpermutation matrix, hence by the same result, it has a partial inverse.  
	Then, we claim that \( \restr{(MN)}{p} = M' \restr{N}{q} \).
	Indeed, let \( x \in \supp p \), \( z \in Z \), \( y \in Y \setminus \supp q \), and suppose that \( M_{x, y} N_{y, z} > 0 \).
	Since \( x \in \supp p \) and \( \restr{M}{p} \) is a subpermutation matrix, we necessarily have that \( M_{x, -} = \dirac{y} \).
	Thus \( q_y = (pM)_y = p_x > 0 \), since \( x \in \supp p \).
	This contradicts \( y \in Y \setminus \supp p \).
	Therefore, \( M_{x, y} N_{y, z} = 0 \).
	Then, 
	\begin{equation*}
		\left(\restr{(MN)}{p}\right)_{x, z} = \sum_{y \in Y} M_{x, y} N_{y, z} = \sum_{y \in \supp q} M_{x, y} N_{y, z} = \left(M' \restr{N}{q}\right)_{x, z}.
	\end{equation*}
	By Lemma \ref{lem:partial_iso_compose} and Lemma \ref{lem:deterministic_iff_quasidirac}, \( \restr{(MN)}{p} \) is the composite of two deterministic partial isomorphisms, thus it is a partial isomorphism as well.
	
	Next, since \( x \in \supp p \) and \( \restr{M}{p} \) is a subpermutation matrix, \( M_{x, -} = \dirac{y} \) for some \( y \in Y \).
	Then, 
	\begin{equation*}
		q_y = \sum_{x' \in X \setminus \set{x} } p_{x'} M_{x', y} + p_x > 0,
	\end{equation*}
	since \( x \in \supp p \).
	Thus, \( y \in \supp q \subseteq \supp N \).
	Then, \( (MN)_{x, -} = N_{y, -} \) is not the zero subdistribution.
	Therefore, \( x \in \supp (MN) \).
	This shows that \( \supp p \subseteq \supp (MN) \) and concludes the proof.
\end{proof}

\begin{lem}
\label{lem:condrev_is_closed_to_tensor}
	Let \( M \) and \( N \) be conditionally reversible computational transformations.
	Then \( M \kro N \) is conditionally reversible.
\end{lem}
\begin{proof}
	First, \( M \kro N \) is deterministic by Lemma \ref{lem:deterministic_tensor}.
	Let \( M \colon p \tf q \) and \( N \colon p' \tf q' \) be the type of \( M \) and \( N \), and denote by \( X, Y, X' \) and \( Y' \) the domains of \( p, q, p' \) and \( q' \) respectively.
	A direct computation shows that \( \supp (p \kro q) = \supp p \times \supp q \).
	Therefore, \( \restr{(M \kro N)}{p \kro q} = \restr{M}{p} \kro \restr{N}{q} \).
	By Lemma \ref{lem:partial_iso_compose} and Lemma \ref{lem:deterministic_iff_quasidirac}, \( \restr{M}{p} \kro \restr{N}{q} \) has a partial inverse.

	Furthermore, by assumption,
	\begin{equation*}
		\supp p \times \supp q \subseteq \supp M \times \supp N.
	\end{equation*}
	To conclude, it is enough to show that the latter is included in \( \supp (M \kro N) \).
	Let \( (x, x') \) in \( \supp M \times \supp N \).
	Then \( (M \kro N)_{xx', -} \) is the distribution defined by \( (y, y') \mapsto M_{x, y} N_{x', y'} \).
	By assumption, there exists \( y \in Y \) and \( y' \in Y' \) such that \( M_{x, y} > 0 \) and \( N_{x', y'} > 0 \), hence their product \( (M \kro N)_{xx', yy'} \) is also \( > 0 \).
	This shows that \( (x, x') \in \supp (M \kro N) \) and concludes the proof.
\end{proof}

\begin{lem} \label{lem:inequality_log}
	Let \( n \geq 2 \), and \( (p_i)_{i = 1}^n \) be a family of real numbers in \( (0, 1) \) such that \( \sum_{i = 1}^n p_i \le 1 \).
	Then 
	\begin{equation*}
		\sum_{i = 1}^n p_i \log(p_i) < \left(\sum_{i = 1}^n p_i\right)\log\left(\sum_{i = 1}^n p_i\right).
	\end{equation*}
\end{lem}
\begin{proof}
	We proceed by induction on \( n \).
	The base case \( n = 2 \) is given by Remark \ref{rmk:log_superadditive}.
	Inductively, let \( n > 2 \), denote by \( p \eqdef \sum_{i = 1}^{n - 1} p_i \le 1 \).
	Then by inductive hypothesis and Remark \ref{rmk:log_superadditive},
	\begin{equation*}
		\sum_{i = 1}^n p_i \log(p_i) < p \log(p) + p_n \log(p_n) < (p + p_n) \log(p + p_n).
	\end{equation*}
	This concludes the proof.
\end{proof}

\begin{prop} \label{prop:condrev_iff_entropy_preserve}
	Let $M \colon p \tf q $ be a deterministic computational transformation. Then, the following are equivalent.
	\begin{enumerate}
	\item $\H(p) = \H(q)$; 
	\item \( M \) is conditionally reversible.
	\end{enumerate}
\end{prop}
\begin{proof} 
	Suppose first that $\H(p) = \H(q)$. 
	We first prove that $\restr{M}{p}$ is a subpermutation matrix, which, by Corollary \ref{cor:partial_inverse_iff_subpermutation_iff_entropy_preserving}, is enough to conclude. 
	Call $A \eqdef (\supp p) \cap (\supp M)$. 
	Since \( M \) is deterministic, we have by Lemma \ref{lem:deterministic_iff_quasidirac} a well defined function \( m \colon A \to Y \) such that for all \( x \in A \), \( M_{x, -} = \dirac{\fun{m}(x)} \).
	For each \( y \in Y \), denote \( A_y \eqdef \invrs{\fun{m}}\set{y} \).
	In order to prove that $\restr{M}{p}$ is a subpermutation matrix, it suffices to prove that $\fun{m}$ is injective, that is, \( A_y \) is either empty or a singleton.
	We let \( Y_0 \coprod Y_1 \coprod Y_{\geq 2} \eqdef Y \), where 
	\begin{align*}
		Y_0 &\eqdef \set{y \in Y \mid A_y \text{ empty}} \\
		Y_1 &\eqdef \set{y \in Y \mid A_y \text{ singleton}} \\
		Y_{\geq 2} &\eqdef Y \setminus (Y_0 \cup Y_1).
	\end{align*}
	If \( y \in Y_1 \), we write \( \invrs{\fun{m}}(y) \) for the unique element \( x \in A_y \) such that \( \fun{m}(x) = y \).
	By means of contradiction, assume that \( Y_{\geq 2} \neq \emptyset \).
	Notice that for all \( y \in Y \), we have 
	\begin{equation*}
		q_y = \sum_{\substack{x \in A \\ \fun{m}(x) = y}} p_x.
	\end{equation*}
	Hence, 
	\begin{align*}
		\H(q) &= - \sum_{y \in Y} \left[ \left( \sum_{x \in A_y} p_x \right) \log \left( \sum_{x \in A_y} p_x \right) \right] \\ 
		&= - \sum_{y \in Y_1} p_{\invrs{\fun{m}}(y)}\log(p_{\invrs{\fun{m}}(y)}) - \sum_{y \in Y_{\geq 2}} \left[\left( \sum_{x \in A_y} p_x \right) \log \left( \sum_{x \in A_y} p_x \right) \right] \\
		&< - \sum_{y \in Y_1} p_{\invrs{\fun{m}}(y)}\log(p_{\invrs{\fun{m}}(y)}) - \sum_{y \in Y_{\geq 2}} \sum_{x \in A_y} p_x \log(p_x) \quad \quad \text{by Lemma \ref{lem:inequality_log}} \\
		&= - \sum_{y \in Y} \sum_{x \in A_y} p_x \log(p_x) \\
		&= - \sum_{x \in A} p_x \log(p_x) \\
		&\le \H(p)
	\end{align*}
	This contradicts the assumption that \( \H(p) = \H(q) \).
	Therefore, \( Y_{\geq 2} = \emptyset \), hence \( \fun{m} \) is injective.
	We now prove that \( B = \supp p \setminus A \) is empty.
	Using a similar computation than above, we find that
	\begin{equation*}
		\H(q) = - \sum_{x \in A} p_x \log(p_x) = \H(p) = - \sum_{x \in A} p_x \log(p_x) - \sum_{x \in B} p_x \log(p_x).
	\end{equation*} 
	This means that \( \sum_{x \in B} p_x \log(p_x) = 0 \).
	Suppose that \( B \) is non-empty.
	Since for all \( x \in B \), \( p_x > 0 \), we have by Lemma \ref{lem:kernel_entropy} that \( B \) is a singleton \( \set{x} \), and \( p = \dirac{x} \).
	Since \( x \notin A \), \( q = M_{x, -} \) is the zero subdistribution, which contradicts the fact that \( q \) is a distribution.
	Therefore \( B \) is empty. 
	Since \( A \subseteq \supp p \), we deduce that \( \supp p = A = \supp p \cap \supp M \).
	This means \( \supp p \subseteq \supp M \). 

	Conversely, suppose that \( M \) is conditionally reversible. 
	Let \( p' \) be the distribution on \( \supp p \) given by \( x \mapsto p_x \).
	Then a direct computation shows \( (p' \restr{M}{p}) = q \).
	Since \( \supp p \subseteq \supp M \), we also have \( \supp p' \subseteq \supp \restr{M}{p} \).
	Thus, by Corollary \ref{cor:partial_inverse_iff_subpermutation_iff_entropy_preserving}, 
	\begin{equation*}
		 \H(q) = \H(p'\restr{M}{p}) = \H(p') = \H(p).
	\end{equation*}
	This concludes the proof.
\end{proof}

\subsection{Equivalence of physical and computational reversibility}

\noindent The following result establishes the notion of equivalence between computational and physical reversibility.

\begin{prop} \label{prop:nee_iff_condrev} 
	Let \( M \) be a physical transformation such that its corresponding computational processes namely \( \Q M \) is deterministic.
	Then, the following are equivalent.
	\begin{enumerate}
	\item $M$ is non-entropy-ejecting.
	\item $\Q M$ is conditionally reversible.
	\end{enumerate}
\end{prop}
\begin{proof}
	Let \( p \tf q \) be the type of \( M \).
	Suppose first that \( M \) is non-entropy-ejecting, then by Lemma \ref{Lem: n.e.e._reversible}, \( \Hcomp(p) \leq \Hcomp(q) \).
	Since \( \Q M \) is deterministic, Lemma \ref{lem:deterministic_iff_entropy_decreasing} implies that \( \Hcomp(q) \leq \Hcomp(p) \).
	Hence \( \Hcomp(p) = \Hcomp(q) \), so by Proposition \ref{prop:condrev_iff_entropy_preserve}, \( \Q M \) is conditionally reversible.

	Conversely, suppose that $\Q M$ is conditionally reversible. 
	Then by Proposition \ref{prop:condrev_iff_entropy_preserve}, \( \H(\Q p) = \H(\Q q) \).
	In particular, \( \Hcomp(p) \le \Hcomp(q) \).
	By Lemma \ref{Lem: n.e.e._reversible}, \( M \) is non-entropy-ejecting.
\end{proof}

As the final step in our compositional modelling framework, we use resource theories to set up the equivalence between computational and physical reversibility as stated in \ref{prop:nee_iff_condrev}.

\begin{dfn}[\tc{Resource theory of physical transformations}] \label{dfn:resource-theory-phy}
	The {\bf resource theory of physical transformations} is the resource theory \( (\Phy, \Phy_{\sf f}) \) whose free transformations are given by the wide subcategory \( \Phy_{\sf f} \) of \( \Phy \) on physical transformations \( M \) that are non-entropy-ejecting and such that their corresponding computational transformations namely \( \Q M \) is deterministic.
\end{dfn}

\begin{dfn}[\tc{Resource theory of computational transformations}] \label{dfn:resource-theory-comp}
	The {\bf resource theory of physical contexts} is the resource theory \( (\Comp, \Comp_{\sf f}) \) whose free transformations are given by the wide subcategory \( \Comp_{\sf f} \) of \( \Comp \) on computational transformations that are conditionally reversible.
\end{dfn}

\begin{rmk}
	By Lemma \ref{lem:det_tot_composition}, Lemma \ref{lem:n.e.e_compose_tensor}, and Lemma \ref{lem:condrev_maps_compose}, \(\Phy_{\sf f} \) and \(  \Comp_{\sf f} \) are well defined (monoidal) subcategories.
\end{rmk}

\begin{thm} [\tc{Fundamental theorem of GRC}] \label{thm: fundamental_theorem}
	The aggregation functor 
	\begin{equation*}
		\Q \colon (\Phy,\Phy_{\sf f}) \to (\Comp, \Comp_{\sf f})	
	\end{equation*}
	is a resource reflecting transformation of resource theories. 
\end{thm}
\begin{proof}
	Immediate by Proposition \ref{prop:nee_iff_condrev}.
\end{proof}

\section{Mathematical backend}
\label{sec: mathematical_backend}

In this section, we prove (Corollary \ref{cor:aggregation_functor}) that the construction sending a partitioned matrix \( M \) to its aggregated matrix \( \Q M \) (Definition \ref{dfn:aggregation}) defines a strong monoidal functor
\begin{equation*}
    \Q \colon (\PSubMat, \kro, \pt) \to (\SubMat, \kro, \pt).
\end{equation*}
We found that this was best explained by Street's formal theory of monads \cite{street1972formal}: to any \( 2 \)\nbd category \( \K \) one can associate a \( 2 \)\nbd category \( \Mnd(\K) \) of monads internal to \( \K \).
Under favorable circumstances, the inclusion 2-functor from \( \K \) to \( \Mnd(\K) \), sending an object to the ``identity monad'' has a left 2-adjoint \( \Kl \colon \Mnd(\K) \to \K \), called the \emph{Kleisli construction}.
Typically, if \( \K \) is the 2-category of categories, functors, and natural transformations, then \( \Mnd(\K) \) is the category whose objects are indeed monads in the usual sense, and the left 2-adjoint computes the usual Kleisli category.

Since \( \SubMat \) and \( \PSubMat \) are the Kleisli categories of the subdistribution monads on sets and partitioned sets respectively, it is reasonable to hope that the functor \( \Q \) will be given by the application of the Kleisli 2-functor on a certain 1-cell in the category \( \Mnd(\K) \).
Since we want \( \Q \) to be a monoidal functor of monoidal categories, we will take \( \K \) to be the 2-category \( \MonCatlax \) of monoidal categories, lax monoidal functors and monoidal transformations.
What remains to be shown is that the subdistribution monads are well-defined objects in \( \Mnd(\MonCatlax) \) (in other words, that the subdistribution monads are \emph{monoidal monads}), and that the aggregation construction defines a \emph{colax morphism of monoidal monads} whose image under the Kleisli 2-functor is indeed \( \Q \).

\subsection{A brief review of the theory of monads and monoids}

\begin{dfn} [\tc{Monad in \( \K \)}]
    Let \( \K \) be a 2-category.
    A {\bf monad in \( \K \)} is given by a 1-cell \( t \colon x \to x \), together with 2-cells \(  \eta \colon 1_x \celto t \) and \( \mu \colon t \comp{0} t \celto t \), satisfying
    \begin{equation*}
        (\eta \comp{0} t) \comp{1} \mu = 1_t = (t \comp{0} \eta) \comp{1} \mu, \quad\text{ and } \quad (\mu \comp{0} t) \comp{1} \mu = (t \comp{0} \mu) \comp{1} \mu.
    \end{equation*}
\end{dfn}

\begin{rmk}
    Using the dual and non-symmetric version of Definition \ref{dfn:commcomonoid_object}, a monad in \( \K \) is a monoid object in the monoidal category \( (\K(x, x), \comp{0}, 1_x) \), 
\end{rmk}

\begin{dfn} [\tc{Colax morphisms}]
    Let \( (t \colon x \to x, \eta, \mu) \) and \( (t' \colon x' \to x', \eta', \mu') \) be two monads in a 2-category \( \K \).
    A {\bf colax morphism between \( (t, \eta, \mu) \) and \( (t', \eta', \mu') \)} is a pair of a 1-cell \( q \colon x \to x' \) and a 2-cell \( \kappa \colon t \comp{0} q \celto q \comp{0} t' \) such that
    \begin{equation*}
        (\eta \comp{0} q) \comp{1} \kappa = q \comp{0} \eta' \quad\text{ and }\quad (\mu \comp{0} q) \comp{1} \kappa = (t \comp{0} \kappa) \comp{1} (\kappa \comp{0} t') \comp{1} (q \comp{0} \mu'). 
    \end{equation*}
\end{dfn}

\noindent Following \cite{street1972formal}, we have the following definition.

\begin{dfn} [\tc{Category of monads}]
    Let \( \K \) be a 2-category.
    The {\bf category of monads in \( \K \)} is the category \( \Mnd(\K) \) whose objects are monads and morphisms are colax morphisms of monads.
\end{dfn}

\begin{rmk}
    In fact, the category \( \Mnd(\K) \) can be naturally promoted to a 2-category using a notion of transformation between colax morphisms of monads.
    Since we will not use this level of generality in the article, we omitted this structure to avoid overwhelming the reader with unnecessary 2-dimensional considerations. 
\end{rmk}

\noindent In the sequel, we will be interested by the notion of monoidal monad, which is a monad in the category of (symmetric) monoidal categories, lax monoidal functors and monoidal transformations.
We recall that given \( (\C, \otimes, e) \) and \( (\C', \otimes, e') \) monoidal categories, a {\bf lax monoidal functor} is a functor \( \F \colon \C \to \C' \), together with 
\begin{itemize}
    \item a structure natural transformation \( t \) of type
        \begin{equation*}
            t_{x, y} \colon \F(x) \otimes \F(y) \to \F(x \otimes y),
        \end{equation*} 
        for all objects \( x, y \) of \( \C \), and
    \item a structure morphism \( u \colon e' \to \F(e) \) in \( \C' \).
\end{itemize}
This data is required to commutes with the associators, unitors, (and braidings) of the (symmetric) monoidal categories in a way that we do not detail here.
The interested reader can consult \cite[XI.2]{maclane1971categories}.
We also recall that a {\bf strong monoidal functor} is a lax monoidal functor such that all the structure morphisms are isomorphisms.

\begin{dfn}[\tc{Monoidal transformation}]
    Let
    \begin{equation*}
        (\F, t, u), (\F', t', u') \colon (\C, \otimes, e) \to (\C', \otimes, e') 
    \end{equation*}
    be two lax monoidal functors.
    A {\bf monoidal transformation from \( \F \) to \( \F' \)} is a natural transformation \( \kappa \colon \F \celto \F' \) such that the diagrams
    \begin{center}
        \begin{tikzcd}
            {\F(x) \otimes \F(y)} & {\F'(x) \otimes \F'(y)} & {\text{and}} && {e'} \\
            {\F(x \otimes y)} & {\F'(x \otimes y)} && {\F(e)} && {\F'(e)}
            \arrow["{\kappa_x \otimes \kappa_y}", from=1-1, to=1-2]
            \arrow["{t_{x, y}}"', from=1-1, to=2-1]
            \arrow["{t'_{x, y}}", from=1-2, to=2-2]
            \arrow["u"', from=1-5, to=2-4]
            \arrow["{u'}", from=1-5, to=2-6]
            \arrow["{\kappa_{x \otimes y}}"', from=2-1, to=2-2]
            \arrow["{\kappa_e}"', from=2-4, to=2-6]
        \end{tikzcd}
    \end{center} 
    commute.
\end{dfn}

\noindent Following \cite{zawadowski2012formal}, symmetric monoidal categories, lax monoidal functors, and monoidal transformations are respectively the 0-cells, 1-cells, and 2-cells of the 2-category \( \MonCatlax \).

\begin{dfn} [\tc{Monoidal monad}] \label{dfn:monoidal_monad}
    A {\bf monoidal monad} is a monad in the 2-category \( \MonCatlax \).
\end{dfn}

\noindent We now recall the classical definition of the Kleisli category.

\begin{dfn} [\tc{Kleisli category}] \label{dfn:Kleisli}
    Let \( (\T, \mu, \eta) \) be a monad on a category \( \C \).
    The {\bf Kleisli category of \( \T \)} is the category \( \Kl(\T) \) whose objects are the same as \( \C \) and whose morphisms from \( x \) to \( y \) are given by morphisms \( f \colon x \to \T(y) \) in \( \cls{D} \).
    The identity is given by the monad unit \( \eta \).
    The Kleisli composition of a composable pair \( (f, g) \) is given by \( f\T(g)\mu \).
\end{dfn}  

\begin{rmk} \label{rmk:Kleisli_2_functor}
    The Kleisli construction extends to a functor
    \begin{equation*}
        \Kl \colon \Mnd(\Cat) \to \Cat,
    \end{equation*}
    by sending in particular a colax morphism of monads \( (\fun{Q}, \kappa) \) from \( \T \) to \( \T' \) to a functor \(  \Kl(\fun{Q}, \kappa) \colon \Kl(\T) \to \Kl(\T') \), given by \( \fun{Q} \) on objects and sending a Kleisli morphism \( f \) from \( x \) to \( y \) to \( f\kappa_y \).

    Furthermore, \( \Kl \) is left adjoint to the inclusion functor \( \iota \colon \Cat \incl \Mnd(\Cat) \) sending a category to the identity monad on it.
    The unit of the adjunction evaluated at a monad \( (\T \colon \C \to \C, \eta, \mu) \) is the usual identity on objects functor \( \T^\kl \colon \C \to \Kl(\T) \) sending a morphism \( f \colon x \to y \) to \( f\eta_y \). 
\end{rmk}

\begin{prop} \label{prop:transfer_monoidal_monad}
    Let \( (\C, \otimes, e) \) be a symmetric monoidal category, \( \T \colon \C \to \C \) be a monoidal monad with structure \( t_{x, x'} \colon \T(x) \otimes \T(x') \to \T(x \otimes x') \).
    Then the Kleisli category of \( \T \) has a symmetric monoidal product defined as follows:
    \begin{enumerate}
        \item the monoidal unit is \( e \);
        \item given two objects \( x \) and \( y \) of \( \Kl(\T) \), then \( x \otimes^\kl y \eqdef x \otimes y \);
        \item given two morphisms \( f \colon x \to y \) and \( f' \colon x' \to y' \) in \( \Kl(\T) \), then
        \begin{equation*}
            f \otimes^\kl f' \eqdef (f \otimes f')t_{y, y'}.
        \end{equation*}
    \end{enumerate} 
    This monoidal structure is such that canonical functor \( \T^\kl \colon \C \to \Kl(\T) \) is strict monoidal.
    Furthermore, if \( \T' \colon \C' \to \C' \) is another monoidal monad on a symmetric monoidal category \( \C' \), and \( (\fun{Q}, \kappa) \) is a colax morphism between the monoidal monad \( \T \) and \( \T' \), then the functor \( \Kl(\fun{Q}, \kappa) \) is lax monoidal with respect to the monoidal structures on \( \Kl(\T) \) and \( \Kl(\T') \), and strong monoidal if \( \fun{Q} \) is strong.
\end{prop}
\begin{proof}
    By \cite[Theorem 6.1]{zawadowski2012formal}, the 2-category \( \MonCatlax \) of symmetric monoidal categories, lax monoidal functors and monoidal transformation admits Kleisli objects.
    This means that the inclusion \( \iota \colon \MonCatlax \incl \Mnd(\MonCatlax) \), sending a symmetric monoidal category to the identity monad, possesses a left adjoint \( \Kl \).
    By that same result, the Kleisli objects are standard, meaning that the functor \( \Kl \) coincides with the functor of Remark \ref{rmk:Kleisli_2_functor} at the level of the underlying categories.
    Using the unit of the adjunction \( \Kl \dashv \iota \), one sees that the \( \T^\kl \) is indeed strict monoidal.
    Finally, since \( \Kl(\fun{Q}, \kappa) \) coincides with \( \fun{Q} \) at the level of objects, if \( \fun{Q} \) is strong, then so is \( \Kl(\fun{Q}, \kappa) \).
\end{proof}

\noindent We conclude this section on the formal theory of monads and monoids with a well-known corollary, which states that the Kleisli category of a monoidal monad on a CDU category is again a CDU category, see \cite[Corollary 3.2]{fritz2020markov}.

\begin{cor} \label{cor:transfer_CDU}
    Let \( (\C, \otimes, e) \) be a CDU category with copy \( \Delta \) and discard \( \varepsilon \), let \( \T \colon \C \to \C \) be a monoidal monad on a symmetric monoidal category \( (\C, \otimes, e) \).
    Then the symmetric monoidal category \( (\Kl(\T), \otimes^\kl, e) \) is a CDU category.
    For all objects \( x \) in \( \C \), the copy morphism is given by \( \Delta_x\eta_x \) and the discard morphism by \( \varepsilon\eta_e \). 
\end{cor}
\begin{proof}
    By Proposition \ref{prop:transfer_monoidal_monad}, the canonical functor \( \T^\kl \colon \C \to \Kl(\T) \) is strict monoidal.
    Therefore, for all objects \( x \) in \( \C \), \( \T^\kl(\Delta_x) \) and \( \T^\kl(\varepsilon_x) \) defines a cocommutative monoid structure in \( \Kl(\T) \), which is uniform.
\end{proof}

\subsection{The Kleisli categories of the subdistribution monads} 

There is a forgetful functor \( \fun{U} \colon \PSet \to \Set \) obtained by forgetting the equivalence relation of a partitioned set.
The functor \( \fun{U} \) has both left and right adjoints defined by assigning to a set \( X \) the equality relation and the relation \( X \times X \), respectively.   

\noindent The following lemma is left as an exercise to the reader. 
\begin{lem} \label{lem:pset_co_complete}
    The category \( \PSet \) is complete and cocomplete.
    Given a diagram \( \fun{F} \colon \cls J \to \PSet \),
    \begin{enumerate}
        \item the limit of \( \fun{F} \) is given by the partitioned set \( (\lim \fun{UF}, \sim) \) where \( \sim \) is the finest equivalence relation on \( \lim \fun{UF} \) making all the projections morphisms of partitioned sets. 
        \item the colimit of \( \fun{F} \) is given by the partitioned set \( (\colim \fun{UF}, \sim) \) where \( \sim \) is the coarsest equivalence relation on \( \colim \fun{UF} \) making all the coprojections morphisms of partitioned sets. 
    \end{enumerate}
\end{lem}

\begin{cor} \label{cor:pset_monoidal}
    The category \( \PSet \) is symmetric monoidal with respect to the cartesian product, with monoidal unit \( \pt \), the terminal set with its unique equivalence relation.
\end{cor}

\begin{dfn} [\tc{Quotient functor}]
    If \( (X, \sim) \) is a partitioned set and \( x \in X \), we write \( \eqclass x \) for the equivalence class of \( x \), and denote by \( \quot X \sim \) the quotient of \( X \) by the equivalence relation \( \sim \).
    This comes equipped with the canonical projection \( \eqclass{-} \colon X \to \quot{X}{\sim} \) sending \( x \) to its equivalence class \( [x] \). 
    The assignment \( (X, \sim) \mapsto \quot{X}{\sim} \) extends to the \emph{quotient functor} \( \Quot \colon \PSet \to \Set \) by sending any \( f \colon (X, \sim) \to (Y, \sim) \) to \( \Quot (f) \colon \quot{X}{\sim} \to \quot{Y}{\sim} \) defined by \( \eqclass{x} \mapsto \eqclass{f(x)} \).    
\end{dfn}

\begin{dfn}[\tc{Subdistribution functor}]    
    Let \( \D X \) be the set of all subdistributions on \( X \). 
    The mapping \( X \mapsto \D X \) extends to a endofunctor \( \D \colon \Set \to \Set \) by mapping any function $f: X \to Y$ and any subdistribution $p$ on \( X \) to the subdistribution
    \begin{equation}
        \D f (p) \colon y \mapsto \sum_{x \in f^{-1}(y)} p(x).
    \end{equation}
\end{dfn}

\noindent By \cite[Definition 2.11]{DiM23}, we may extend \( \D \) to a monad on \( \Set \) as follows.
The name is justified by the fact that the Kleisli of this monad is the category \( \SubMat \) (Proposition \ref{prop:Kleisli_is_matrix}).

\begin{dfn} [\tc{Subdistribution monad on Set}] \label{dfn:subdist_set}
\label{Defn:subdistribution-monad}
    We define the natural transformations 
    \begin{equation*}
        \eta \colon \bigid{} \Rightarrow \D,\quad\quad \mu \colon \D\D \Rightarrow \D,
    \end{equation*}
    where for \( X \in \Set \), \( \eta_X \) maps \( x \in X \) to \( \dirac x \) and \( \mu_X \) maps \( \sigma \in \D\D X \) to the subdistribution
    \begin{equation*}
        x \mapsto \sum_{p \in \D X} \sigma(p) p_x.
    \end{equation*}
    The monad \( (\D, \mu, \eta) \) is called the {\bf subdistribution monad on \( \Set \)}.
\end{dfn}

\begin{lem} \label{lem:subdistribution_monoidal}
    The monad \( \D \) is monoidal with respect to the cartesian product on \( \Set \), with structural transformation \( t_{X, Y} \colon \D X \times \D Y \to \D(X \times Y) \) defined, for all \( (\sigma, \tau) \in \D X \times \D Y \), by
    \begin{equation*}
        t_{X, Y}(\sigma, \tau) \colon (x, y) \mapsto \sigma(x) \tau(y),
    \end{equation*}
    and \( u \colon \pt \to \D \pt \) classifying the unique distribution on \( \pt \). 
\end{lem}
\begin{proof}
    That \( \D \) defines a lax monoidal functor follows readily from associativity and unitality of multiplication of real numbers.
    Then, given \( (x, y) \in X \times Y \), we have \( t_{X, Y}(\dirac{x}, \dirac{y}) = \dirac{(x, y)} \) and \( u = \eta_\pt \).
    Thus, \( \eta \) is a monoidal transformation.
    Next, to show that \( \mu \) is a monoidal transformation, we must first show that
    \begin{center}
        \begin{tikzcd}
            {\D\D X \times \D \D Y} && {\D X \times \D Y} \\
            {\D(\D X \times \D Y)} \\
            {\D \D (X \times Y)} && {\D(X \times Y)}
            \arrow["{\mu_X \times \mu_Y}", from=1-1, to=1-3]
            \arrow["{t_{\D X, \D Y}}"', from=1-1, to=2-1]
            \arrow["{t_{X, Y}}", from=1-3, to=3-3]
            \arrow["{\D(t_{X, Y})}"', from=2-1, to=3-1]
            \arrow["{\mu_{X \times Y}}"', from=3-1, to=3-3]
        \end{tikzcd}
    \end{center}
    commutes.
    A careful diagram chasing reveals that the lower path sends \( (\sigma, \tau) \) to the subdistribution on \( X \times Y \) given by
    \begin{equation*}
        (x, y) \mapsto \sum_{s \in \D(X \times Y)} \left(\sum_{(p, q) \in \invrs{t}_{X, Y}(s)} \sigma(p)\tau(q) \right) s(x, y),
    \end{equation*}
    which is the same as
    \begin{equation*}
       (x, y) \mapsto \sum_{(p, q) \in \D X \times \D Y} \sigma(p)p_x \tau(q)q_y,
    \end{equation*}
    which is exactly the subdistribution given by applying the upper path to \( (\sigma, \tau) \).
    The final axiom follows from the monad axiom and the fact that \( u = \eta_\pt \).
    This concludes the proof. 
\end{proof}

\begin{dfn} [\tc{Lifting from \( \Set \) to \( \PSet \)}]
    Let \( \F \colon \Set \to \Set \) be an endofunctor of \( \Set \) and suppose that an assignment \( (X, \sim) \mapsto (\F X, \sim) \) mapping a partitioned set \( X \) to an equivalence relation on \( \F X \) is fixed.
    We say that \( \F \) {\bf lift to a functor of partitioned sets} if for all functions \( f \colon (X, \sim) \to (Y, \sim) \) of partitioned sets, \( \F f \colon \F X \to \F Y \) is a function of partitioned set \( \F f \colon (\F X, \sim) \to (\F Y, \sim) \).
    In that case, this defines a unique endofunctor \( \tilde{\F} \colon \PSet \to \PSet \) such that \( \fun{U}\tilde{\F} = \F \).
    
    Furthermore, if \( \G \colon \Set \to \Set \) is another endofunctor of \( \Set \) lifting via an assignment \( (X, \sim) \mapsto (\G X, \sim) \) to an endofunctor \( \tilde{\G} \colon \PSet \to \PSet \), and \( \alpha \colon \F \celto \G \) is a natural transformation, we say that \( \alpha \) lifts to a natural transformation on \( \PSet \) if for all partitioned sets
     \( (X, \sim) \), the function \( \alpha_X \colon \F X \to \G X \) is a function of partitioned sets \( \alpha_{(X, \sim)} \colon (\F X, \sim) \to (\G X, \sim) \).
    In that case, this defines a unique natural transformation \( \tilde{\alpha} \colon \tilde{\F} \celto \tilde{\G} \) such that \( \fun{U}\tilde{\alpha} = \alpha \).
\end{dfn}

\noindent To avoid overloading the notation in the next proofs, we will employ the following shorthand.
\begin{dfn} [Implicit notation for summation] \label{dfn:implicit_summation}
	Let \( p \colon Y \to \mathbb{R} \) be a finitely supported function. 
	We write 
	\begin{equation*}
		\sum_Y p \eqdef	\sum_{y \in Y} p(y).
	\end{equation*}
\end{dfn}  

\noindent By Definition \ref{dfn:equivalent_subdistribution}, given a partitioned set \( (X, \sim) \), we may endow \( \D X \) with the equivalence relation given by equivalence of subdistributions.
In the next Lemma, we show that this extension is compatible with the functorial action of \( \D \) as well as with the monad structure.

\begin{prop} \label{prop:D_lift_to_monad_on_pset}
    The assignment \( (X, \sim) \mapsto (\D X, \sim) \) lifts the subdistribution monad \( (\D, \eta, \mu) \) on \( \Set \) to a monad \( (\tD, \tilde{\eta}, \tilde{\mu}) \) on \( \PSet \), which is monoidal with respect to the cartesian product on partitioned sets.
\end{prop}
\begin{proof}
	Let $f \colon (X, \sim) \to (Y, \sim)$ be a morphism of partitioned sets, and let \( p, q \in \D X \) such that $p \sim q$.
    For any \( x, x' \in X \) such that $x \sim x' \in X$, we have that $f(x) \sim f(x')$.
    Thus, for all \( y \in Y \), the set \( \invrs{f}\eqclass{y} \) is either empty or of the form \( \eqclass{f(x)} \) for any \( x \in \invrs{f}(y) \). 
    Therefore, for all \( y \in Y \)
    \begin{equation*}
        \sum_{\eqclass{y}} \D f(p) = \sum_{\invrs{f}\eqclass{y}} p = \sum_{\invrs{f}\eqclass{y}} q = \sum_{\eqclass{y}} \D f(q),
    \end{equation*}
    where the in the middle step we used \( p \sim q \).
    This shows that \( \D f(p) \sim \D f(q) \), hence that \( \D \) lifts to an endofunctor \( \tD \colon \PSet \to \PSet \). 
        
    Next, let \( (X, \sim) \) be a partitioned set. 
    Then it is straightforward to check that for all $x, x' \in X$, if \( x \sim x' \) then $\dirac{x} \sim \dirac{x'}$, thus \( \eta \) lifts to a natural transformations \( \tilde{\eta} \). 
    Finally, let \( \sigma, \tau \in \D\D X \) such that \( \sigma \sim \tau \) and \( x \in X \).
    We have
    \begin{align*}
        \sum_{\eqclass{x}}\mu_X(\sigma) &= \sum_{p \in \D X} \sum_{x' \sim x} p(x') \sigma(p) \\ 
        &= \sum_{\eqclass{q} \in \quot {\D X} \sim} \sum_{p \in \eqclass{q}} \sigma(p) \sum_{x' \sim x} p(x') \\
        &= \sum_{\eqclass{q} \in \quot {\D X} \sim} \sum_{p \in \eqclass{q}} \sigma(p) \sum_{x' \sim x} q(x') && \text{since } p \sim q \\
        &= \sum_{\eqclass{q} \in \quot {\D X} \sim} \left(\sum_{x' \sim x} q(x')\right) \sum_{p \in \eqclass{q}} \sigma(p)  \\
        &= \sum_{\eqclass{q} \in \quot {\D X} \sim} \left(\sum_{x' \sim x} q(x')\right) \sum_{p \in \eqclass{q}} \tau(p) && \text{since } \sigma \sim \tau \\
        &= \sum_{\eqclass{x}}\mu_X(\tau).
    \end{align*}
    Therefore, \( \mu_X(\sigma) \sim \mu_X(\tau) \), showing that \( \mu \) lifts to a natural transformation \( \tilde{\mu} \).
    Since \( \fun{U} \) is faithful, the monad axioms for \( (\tD, \tilde{\eta}, \tilde{\mu}) \) hold automatically.

    To show that \( \tD \) is monoidal, it is enough by Lemma \ref{lem:subdistribution_monoidal}, and since \( \fun{U} \) is faithful, to show that for all partitioned sets \( (X, \sim) \) and \( (Y, \sim) \), the structural function \( t_{X, Y} \) defines a morphism of partitioned sets 
    \begin{equation*}
        t_{(X, \sim), (Y, \sim)} \colon \tD (X, \sim) \times \tD(Y, \sim) \to \tD ((X \times Y), \sim).
    \end{equation*}
    Let \( (\sigma, \tau) \sim (\sigma', \tau') \) in \( \tD X \times \tD Y \). 
    By Lemma \ref{lem:pset_co_complete}, this is equivalent to say that \( \sigma \sim \sigma' \) and \( \tau \sim \tau' \).
    By the same result, for any couple \( (x, y) \in X \times Y \), we have \( \eqclass{(x, y)} = \eqclass{x} \times \eqclass{y} \).
    Therefore, 
    \begin{align*}
        \sum_{\eqclass{(x, y)}} t(\sigma, \tau) &= \sum_{\eqclass{x}} \sigma \sum_{\eqclass{y}} \tau \\
        &=  \sum_{\eqclass{x}} \sigma' \sum_{\eqclass{y}} \tau' && \text{since } \sigma \sim \sigma' \text{ and }\tau \sim \tau' \\
        &= \sum_{\eqclass{(x, y)}} t(\sigma', \tau').
    \end{align*} 
    This means that \( t(\sigma, \tau) \sim t(\sigma', \tau') \) and concludes the proof.
\end{proof}

\noindent Using the previous result, we make the following definition.

\begin{dfn} [\tc{Subdistribution monad on partitioned sets}] \label{dfn:subdistribution_pset}
    The {\bf subdistribution monad on partitioned sets} is the monoidal monad \( (\tD, \tilde{\eta}, \tilde{\mu}) \) on \( \PSet \) obtained by lifting the subdistribution monad on \( \Set \).
\end{dfn}

\noindent For each partitioned set \( (X, \sim) \), there is a function \( \quottf_{(X, \sim)} \colon \quot{(\D X)}{\sim} \to \D(\quot X \sim) \)    defined by 
    \begin{equation*}
     \quottf_{(x, \sim)}(\eqclass{p}) \eqdef \eqclass{x} \mapsto \sum_{\eqclass{x}} p, 
     \end{equation*} which does not depend, by Definition \ref{dfn:equivalent_subdistribution}, on the chosen representative \( p \) of \( \eqclass{p} \).
    This defines a natural transformation \( \quottf \) fitting in
    \begin{center}
        \begin{tikzcd}
            \PSet & \PSet \\
            \Set & {\Set.}
            \arrow["\tD", from=1-1, to=1-2]
            \arrow["\Quot"', from=1-1, to=2-1]
            \arrow["\quottf"', shorten <=2pt, shorten >=2pt, Rightarrow, from=1-2, to=2-1]
            \arrow["\Quot", from=1-2, to=2-2]
            \arrow["\D"', from=2-1, to=2-2]
        \end{tikzcd}
    \end{center}

\begin{lem} \label{lem:quottf_colax_morphism}
    The pair \( (\Quot, \quottf) \) defines a colax morphism between the monoidal monads \( \tD \) and \( \D \).
\end{lem}
\begin{proof}
    The functor \( \Quot \) is a strong monoidal functor, hence is in particular lax monoidal.
    We show that the natural transformation \( \quottf \) is a monoidal transformation. 
    The first step is to show that the diagram
    \begin{center}
        \begin{tikzcd}
            {\quot{(\D X)}{\sim} \times \quot{(\D Y)}{\sim}} & {\D(\quot X \sim) \times \D(\quot Y \sim)} \\
            {\quot{(\D X \times \D Y)}\sim} & {\D((\quot X \sim) \times (\quot Y \sim))} \\
            {\quot{(\D(X \times Y))}\sim} & {\D(\quot {(X \times Y)}\sim)}
            \arrow["{\quottf \times \quottf}", from=1-1, to=1-2]
            \arrow["\cong"{description}, from=1-1, to=2-1]
            \arrow["{t_{\quot X \sim, \quot Y \sim}}", from=1-2, to=2-2]
            \arrow["{\Quot(\tilde t_{X, Y})}"', from=2-1, to=3-1]
            \arrow["\cong"{description}, from=2-2, to=3-2]
            \arrow["\quottf"', from=3-1, to=3-2]
        \end{tikzcd}
    \end{center}
    commutes, which is immediate by diagram chasing.
    The other axiom is equally straightforward.
    We now show that \( (\Quot, \quottf) \) is a colax morphism of monads.
    Let \( (X, \sim) \) be a partitioned set. 
    By inspection, the unit axiom amounts to showing that for all \( x \in X \), \( \dirac{\eqclass{x}} \) is equal to
    \begin{equation*}
            \quottf(\eqclass{\dirac{x}}) = \eqclass{x'} \mapsto \sum_{\eqclass{x'}} \dirac{x}, 
    \end{equation*}
    which is again a direct check.
    For the multiplication axiom, one need to prove that the diagram, where indices on \( \quottf \) and \( \mu \) have been left implicit, 
    \begin{center}
        \begin{tikzcd}
            {\quot{(\D\D X)} \sim} && {\quot{(\D X)} \sim} \\
            {\D(\quot {(\D X)} \sim)} \\
            {\D\D(\quot X \sim)} && {\D(\quot X \sim)},
            \arrow["{\Quot(\mu)}", from=1-1, to=1-3]
            \arrow["\quottf"', from=1-1, to=2-1]
            \arrow["\quottf", from=1-3, to=3-3]
            \arrow["\D(\quottf)"', from=2-1, to=3-1]
            \arrow["\mu"', from=3-1, to=3-3]
        \end{tikzcd}
    \end{center}
    commutes.
    A precise examination of the definitions reveals that the lower paths sends \( \eqclass{\sigma} \in \quot {(\D \D X)} \sim \) to 
    \begin{equation*}
        \eqclass{x} \mapsto \sum_{\tilde p \in \D (\quot X {\scriptstyle\sim})} \left(\sum_{\eqclass{q} \in \invrs{\quottf}(\tilde{p})} \sum_{\eqclass{q}} \sigma\right)\tilde{p}_{\eqclass{x}}.
    \end{equation*}
    Now, observe that \( \eqclass{q} \in \invrs{\quottf}\tilde{p} \) if and only if \( \quottf\eqclass{q} = \tilde{p} \) if and only if for all \( \eqclass{x} \in \quot X \sim \), we have \( \tilde{p}_{\eqclass{x}} = \sum_{\eqclass{x}} q. \)
    Thus, for all \( \eqclass{x} \in \quot X \sim \), we get
    \begin{align*}
       \sum_{\tilde p \in \D (\quot X {\scriptstyle\sim})} \left(\sum_{\eqclass{q} \in \invrs{\quottf}(\tilde{p})} \sum_{\eqclass{q}} \sigma\right)\tilde{p}_{\eqclass{x}} &=
        \sum_{\tilde p \in \D (\quot X {\scriptstyle\sim})} \sum_{\eqclass{q} \in \invrs{\quottf}(\tilde{p})} \sum_{\eqclass{q}} \sigma \sum_{\eqclass{x}} q \\
        &= \sum_{\eqclass{q} \in \quot {(\D X)} {\scriptstyle\sim}}  \sum_{\eqclass{q}} \sigma \sum_{\eqclass{x}} q \\
        &= \sum_{\eqclass{q} \in \quot {(\D X)} {\scriptstyle\sim}}  \sum_{p \sim \eqclass{q}} \sigma(p) \sum_{x' \sim \eqclass{x}} p_{x'} & \text{since } p \sim q, \\
        &= \sum_{x' \sim \eqclass{x}} \sum_{p \in \D X} \sigma(p)p_{x'} \\
        &= \sum_{\eqclass{x}} \mu(\sigma) \\
        &= \quottf \eqclass{\mu(\sigma)} \eqclass{x} = \quottf\left( \Quot(\mu)(\eqclass{\sigma}) \right).
    \end{align*}
    This shows that the diagram commutes and concludes the proof.
\end{proof}

\noindent By Proposition \ref{prop:transfer_monoidal_monad} and Lemma \ref{lem:subdistribution_monoidal}, we have the following monoidal structure on \( \Kl(\D) \).

\begin{dfn} [\tc{Monoidal structure on \( \Kl(\D) \)}] \label{dfn:Kleisli_monoidal_distribution}
    The cartesian product in \( \Set \) sets induces a monoidal structure \( (\Kl(\D), \otimes, \pt) \) on \( \Kl(\D) \) defined by the cartesian product on objects, and, given two morphisms \( f \colon X \to Y \) and \( f' \colon X' \to Y' \) in \( \Kl(\D) \), 
    \begin{equation*}
        f \otimes f' \eqdef (f \times f')t_{Y, Y'}.
    \end{equation*}
\end{dfn}

\noindent Similarly, by Proposition \ref{prop:transfer_monoidal_monad} and Proposition \ref{prop:D_lift_to_monad_on_pset}, we have the following monoidal structure on \( \Kl(\tD) \).

\begin{dfn} [\tc{Monoidal structure on \( \Kl(\tD) \)}]
    The cartesian product in \( \PSet \) sets induces a monoidal structure \( (\Kl(\tD), \otimes, \pt) \) on \( \Kl(\tD) \) defined by the cartesian product on objects, and, given two morphisms \( f \colon (X, \sim) \to (Y, \sim) \) and \( f' \colon (X', \sim) \to (Y', \sim) \) in \( \Kl(\tD) \), 
    \begin{equation*}
        f \otimes f' \eqdef (f \times f')\tilde{t}_{(Y, \sim), (Y', \sim)}.
    \end{equation*}
\end{dfn}

\begin{thm} \label{thm:induced_morphism_into_Kleisli}
    The monoidal functor \( \Quot \colon (\PSet, \times, \pt) \to (\Set, \times, \pt) \) extends to a strong monoidal functor 
    \begin{equation*}
        \Q \colon (\Kl(\tD), \otimes, \pt) \to (\Kl(\D), \otimes, \pt)
    \end{equation*}
    by sending a morphism \( f \colon (X, \sim) \to (\D Y, \sim) \) to \( \Q(f) \eqdef \Quot(f) \quottf_{(Y, \sim)} \). 
\end{thm}
\begin{proof}
    By Lemma \ref{lem:quottf_colax_morphism}, the pair \( (\Quot, \quottf) \) defines a colax morphism in the category \( \Mnd(\MonCatlax) \).
    The result then follows from Proposition \ref{prop:transfer_monoidal_monad} and the fact that \( \Quot \) is strong monoidal.
\end{proof}

\subsection{Kleisli categories, subdistribution and partitioned matrices}
\label{Kleisli-of-subdistribution-monads}

Recall the categories \( \SubMat \) and \( \PSubMat \) from Definition \ref{dfn:subdistribution_matrices} and Definition \ref{dfn:partionned_matrices}.
In the former, the objects are sets, morphisms are subdistribution matrices, and composition is given by matrix multiplication.
In the latter, the objects are partitioned sets, morphisms are subdistribution matrices, and composition is given again by matrix multiplication.

\begin{prop} \label{prop:Kleisli_is_matrix}
    The categories \( \Kl(\D) \) and \( \Kl(\tD) \) are respectively equivalent to the categories \( \SubMat \) and \( \PSubMat \).
    The monoidal structure on the Kleisli categories transfers to a monoidal structure on \( \SubMat \) and \( \PSubMat \) whose monoidal product is given by the Kronecker product and whose monoidal unit is the terminal object.
\end{prop}
\begin{proof}
    By currying, a subdistribution matrix \( M \) of shape \( (X, Y) \) is equivalently a function \( M \colon X \to \D Y \). 
    We leave to the reader the exercise of unfolding the definition of Kleisli composition and check that it coincides with matrix multiplication, and similarly, that the Kronecker product is curryied to Definition \ref{dfn:Kleisli_monoidal_distribution}.
\end{proof}

\noindent Recall the copy and discard matrices from Definition \ref{dnf:copy_discard_sustoch}.

\begin{cor} \label{cor:submat_pmat_copy_discard}
    The categories \( \SubMat \) and \( \PSubMat \) are CDU categories, with copy-discard structure given by the copy matrix and the discard matrix.
\end{cor}
\begin{proof}
    By Corollary \ref{cor:transfer_CDU} under the equivalence of categories given by Proposition \ref{prop:Kleisli_is_matrix}. 
\end{proof}

\noindent By the last result and Theorem \ref{thm:induced_morphism_into_Kleisli}, we obtain a monoidal functor 
\begin{equation*}
    \Q \colon (\PSubMat, \kro, \pt) \to (\SubMat, \kro, \pt).
\end{equation*}
unfolding the definition, \( \Q \) is defined by sending a partitioned set \( (X, \sim) \) to \( \quot{X}{\sim} \) and a partitioned matrix \( M \) of shape \( (X, Y) \) to the matrix of shape \( (\quot{X}{\sim}, \quot{Y}{\sim}) \) given by
\begin{equation*}
    (\eqclass{x}, \eqclass{y}) \mapsto \sum_{[y]} M_{x, -}.
\end{equation*}
This is the aggregated matrix of Definition \ref{dfn:aggregation}.
Therefore, we conclude with the following corollary.

\begin{cor} \label{cor:aggregation_functor}
    The construction sending a partitioned matrix \( M \) to its aggregated matrix \( \Q M \) defines a strong monoidal functor
    \begin{equation*}
        \Q \colon (\PSubMat, \kro, \pt) \to (\SubMat, \kro, \pt).
    \end{equation*}
\end{cor}
\begin{proof}
    By Theorem \ref{thm:induced_morphism_into_Kleisli} under the equivalence of categories given by Proposition \ref{prop:Kleisli_is_matrix}.
\end{proof}

\section{Summary and future work}

In this work, we developed a compositional theory for generalized reversible computing and formalized its fundamental theorem \cite[Theorem 6]{Fra18} as a resource-reflecting functor. We introduced the notion of partitioned matrices and proved that the category of partitioned matrices is a Kelisli category of the subdistribution monad on the category of partitioned sets and equivalence-preserving functions. 
We constructed a functor from the category of partitioned matrices to the category of subdistribution matrices, called the aggregation functor, and proved it to be strong monoidal. Using this framework as a foundation, we set up the notions of physical reversibility using non-entropy ejecting processes and logical reversibility using conditional reversibility. 
Finally, in Proposition \ref{prop:nee_iff_condrev}, we proved the equivalence between physical and computational reversibility: a physical transformation whose corresponding computational transformation is deterministic, is non-entropy ejecting if and if only if the computational transformation conditionally reversible. This result validates the logical consistency of Landauer's principle within our categorical semantics for generalized reversible computing.

The category of partitioned matrices provides a formal probabilistic language to model microcosm (physical phenomena) nested within macrocosm (computational structure), which is a recurring theme in physics, for example, in thermodynamics, generalized probabilistic theories, and in quantum dynamics.
As such applying the compositional framework developed in this paper to other probabilistic settings in physics is a topic for future exploration. 

Furthermore, applying this framework to modeling physical reversible computing circuits is yet to be explored. Since the theory accommodates conditional reversibility, it can formally model engineering design choices for gates where reversibility is conditional on specific inputs.

\printbibliography

\end{document}